\documentclass[sn-mathphys-num]{sn-jnl}

\usepackage{graphicx}%
\usepackage{multirow}%
\usepackage{amsmath,amssymb,amsfonts}%
\usepackage{amsthm}%
\usepackage{mathrsfs}%
\usepackage[title]{appendix}%
\usepackage{xcolor}%
\usepackage{textcomp}%
\usepackage{manyfoot}%
\usepackage{booktabs}%
\usepackage{algorithm}%
\usepackage{algorithmicx}%
\usepackage{algpseudocode}%
\usepackage{listings}%

\theoremstyle{thmstyleone}%
\newtheorem{theorem}{Theorem}

\newtheorem{lemma}[theorem]{Lemma}%

\theoremstyle{thmstyletwo}%
\newtheorem{remark}{Remark}%

\theoremstyle{thmstylethree}%

\begin{document}

\title[A moving mesh Discontinuous Galerkin method]{A velocity-based moving mesh Discontinuous Galerkin method for the advection-diffusion equation}

\author*[1]{\fnm{Ezra} \sur{Rozier}}\email{ezra.rozier@inria.fr}

\author[2]{\fnm{J\"{o}rn} \sur{Behrens}}\email{joern.behrens@uni-hamburg.de}

\affil*[1]{\orgdiv{Centre INRIA de l'Universit\'{e} de Rennes}, \orgname{INRIA}, \orgaddress{\street{263 Avenue du G\'{e}n\'{e}ral de Gaulle}, \city{Rennes}, \postcode{35000}, \country{France}}}

\affil[2]{\orgdiv{Department of Mathematics}, \orgname{University of Hamburg}, \orgaddress{\street{Bundesstrasse 55}, \city{Hamburg}, \postcode{20146}, \country{Germany}}}

\abstract{In convection-dominated flows, robustness of the spatial discretisation is a key property. While Interior Penalty Galerkin (IPG) methods already proved efficient in the situation of large mesh Peclet numbers,  Arbitrary Lagrangian-Eulerian (ALE) methods are able to reduce the convection-dominance by moving the mesh. In this paper, we introduce and analyse a velocity-based moving mesh discontinuous Galerkin (DG) method for the solution of the linear advection-diffusion equation. By introducing a smooth parameterized velocity $\Tilde{\textbf{V}}$ that separates the flow into a mean flow, also called moving mesh velocity, and a remaining advection field $\textbf{V}-\Tilde{\textbf{V}}$, we made a convergence analysis based on the smoothness of the mesh velocity. Furthermore, the reduction of the advection speed improves the stability of an explicit time-stepping. Finally, by adapting the existing robust error criteria to this moving mesh situation, we derived robust \textit{a posteriori} error criteria that describe the potentially small deviation to the mean flow and include the information of a transition towards $\textbf{V}=\Tilde{\textbf{V}}$.
}

\keywords{discontinuous Galerkin method, a posteriori error bound, a priori error bound, moving mesh}


\pacs[MSC Classification]{65N15, 65N30}

\pacs[Acknowledgements]{The authors acknowledge the support by the Deutsche Forschungsgemeinschaft (DFG) within the Research Training Group GRK 2583 "Modeling, Simulation and Optimization of Fluid Dynamic Applications”.}

\maketitle

\section{Introduction}\label{sec:intro}

Approaches to approximate advection phenomena with a moving mesh have been explored in several forms for a long time (e.g., \cite{hirt74} from 1974). The advection-dominance disappears and the scheme gains stability due to the reduction of the Courant number. In this landscape, the ALE approach that is based on a dynamic transformation map, and that we can see at work for finite elements in the Lagrange-Galerkin method (see \cite{bause02}, \cite{chrysafinos06}, \cite{chrysafinos08}) or with other forms of spatial discretisation, for finite volumes (\cite{klingenberg15}) or DG methods (\cite{zhou19}). In the past, the ALE approach has been combined successfully with error analysis (\cite{bause02}, \cite{chrysafinos08}). Furthermore, the ALE-DG approach, which provides conservation properties \cite{cslam}, has already been discussed for problems with nonlinear diffusion (\cite{balaszova18}) or in semi-Lagrangian methods for linear advection (\cite{restelli06}) and advection-diffusion (\cite{bonnaventura21}, \cite{cai20}) where the precision largely depends on the regularity of the advection velocity.

However, even after the advection-dominance has been removed or reduced, the presence of a small diffusion term and consequently a big mesh Peclet number $\frac{h}{\varepsilon}$ means that the space semi-discretisation still needs to be robust. Therefore, we need to preserve the robustness properties found in IPG (\cite{wheeler79}, \cite{cangiani17} for complex geometries and diffusion tensor) or in its weighted version (\cite{ern05}, \cite{burman06}). These approaches have given rise to robust \textit{a posteriori} error estimates in order to develop adaptive mesh refinement (AMR) for the steady-state (\cite{schotzau09}) or for the unsteady case (\cite{cangiani14}). These error estimates (see initially \cite{schotzau09}) are based, for the steady-state problem, on remainder-based criteria and are proven through elliptic reconstruction.

In this paper we build a velocity-based moving mesh DG method that shows that taking precise account of what the deformation map does to the elliptic term preserves the convergence and robustness properties of the IPG. Since the smoothness of the deformation map is important in order to ensure the scheme's convergence, we separate the advection field into a smooth mean flow with which the map moves and a remaining advection field that contains the small-scale deviations from the mean flow. 

Splitting the advection field helps us to use a more regular moving mesh velocity that avoids the strong deformations that create errors in the existing ALE approaches (for instance entanglement), furthermore it preserves the ellipticity of the problem thus resulting both in the convergence of the IPG formulation and the use of remainder-based refinement criteria. Additionally this splitting makes possible an arbitrarily high resolution of the moving mesh's ordinary differential equation with no impact on the degrees of freedom of the DG method. Finally, the remaining advection field, even with strong deformation (large spatial derivative), can still remain small, resulting in a larger convergence rate than that of the DG method on a static mesh.

The \textit{a priori} study of the error estimate shows that when possible, the maximal reduction of the advection term saves half an order of convergence for the spatial semi-discretisation. The \textit{a posteriori} error estimate allows us to focus more closely on the elliptical term and to have a robust approach to the error. Finally, the test case involving a boundary layer problem shows that the ALE method makes it possible to get rid of the advection-dominance and focus on small-scale effects. This would not be possible with a static mesh method.

We consider the unsteady scalar advection-diffusion equation:

\begin{equation}\label{eq:c-d}
\begin{array}{ll}
     \frac{\partial u}{\partial t} + \textbf{V} \cdot \nabla u - \varepsilon \Delta u = f \hspace{5mm} & [0,T] \times \Omega \\
     u = u_D & [0,T] \times \Gamma_D \\
     \varepsilon \frac{\partial u}{\partial \textbf{n}} = u_N & [0,T] \times \Gamma_N \\
     u(x,0) = u_0 (x) & \Omega
\end{array}
\end{equation}

\noindent in a bounded space-time cylinder with a convex cross-section $\Omega \subset \mathbb{R}^2$, having a Lipschitz boundary $\Gamma$ consisting of two disjoint connected parts $\Gamma_D$ and $\Gamma_N$. The final time $T$ is arbitrary, but kept fixed in what follows. We assume that the data satisfy the following conditions

\begin{align}
    &f \! \in\!  C(0,T;L^2(\Omega)), \;  u_N\!  \in \! C(0,T;L^2(\Gamma_N)), \; \textbf{V} \! \in\!  C(0,T;W^{1,\infty}(\Omega))^2. \label{ass:A1} \\ 
    & u_D \mbox{ is the trace of some } u^*\in C(0,T;H^1(\Omega))\cap L^\infty(0,T;L^\infty(\Omega)) \mbox{ on } \Gamma_D\times[0,T] \label{ass:A2}\\
    & \forall t \in [0,T], \; \{ x \in \Gamma \colon  \textbf{V}(t,x) \cdot \textbf{n}(x) < 0\} \subset \Gamma_D \label{ass:A4}\\ 
    & 0 < \varepsilon  \label{ass:A5}\\
    & u_0 \in L^2(\Omega).\label{ass:A6}
\end{align}

\noindent In this formulation, $\textbf{V}$ is a prescribed velocity (for instance if \eqref{eq:c-d} is the equation for the concentration of a chemical species, then $\textbf{V}$ is the velocity of the solute).

\noindent Writing $\mathcal{Q}^T=[0,T]\times \Omega$, for $p\geq 1$ and $k\in \mathbb{N}$, we define the spaces $L^p(0,T;Z)$, respectively $C^k(0,T;Z)$ (with $Z$ a Banach space of measurable functions with values in $\Omega$), that consist of functions $v \colon \mathcal{Q}^T  \rightarrow \mathbb{R}$ for which:

\begin{align}
    &\lVert v \lVert_{L^p(0,T;Z)}^p = \int_0^T \lVert v(t,\cdot)\rVert _Z^p d t < +\infty \mbox{ for } 1\leq p< +\infty \label{LpLpOmega}\\
    &\lVert v \rVert_{L^\infty(0,T;Z)} = \underset{0\leq t \leq T}{\mbox{ess sup}} \lVert v(t,\cdot) \rVert _Z <+\infty \mbox{ for } p=+\infty \nonumber
\end{align}

\noindent Respectively $ \exists w\colon [0,T] \rightarrow Z$ s.t. $w(t)=v(t,\cdot)$ and $w\in C^k([0,T])$. For a Sobolev space $Z(\Omega)$ we will have $Z_p(\mathcal{Q}^T)=L^p(0,T;Z(\Omega))$ and $Z_{C^k}(\mathcal{Q}^T)=C^k(0,T;Z(\Omega))$. For instance $H^1_p(\mathcal{Q}^T)=L^p(0,T;H^1(\Omega))$ and $W^{1,\infty}_C(\mathcal{Q}^T)=C(0,T;W^{1,\infty}(\Omega))$. Finally the norm associated with $Z_p(\mathcal{Q}^T)^2$ is $\lVert \; \lVert \;\lVert \cdot \lVert_Z \lVert_{l^2(\{1;2\})} \lVert_{L^p([0,T])}$.

This equation will be reformulated with the help of a flowmap in section \ref{sec:flowmaps}. In section \ref{sec:sd} we introduce the DG semi-discretisation of equation \eqref{eq:mmform}, followed in section \ref{sec:errest} by a study of an \textit{a priori} error estimate in order to discuss the convergence of the method and an \textit{a posteriori} error estimate in order to discuss the robustness, paving the way for AMR. Section \ref{sec:testcase} presents a boundary layer problem in order to demonstrate the properties of the method. Finally section \ref{sec:conc} presents our conclusions.

In what follows, we will consider the case where $u_D=0$. Consequently, $u^*=0$. For the case $u_D \neq 0$, we refer to remark \ref{rq:jump}.

\section{Flow maps}
\label{sec:flowmaps}
We consider the ALE-DG method as a classical DG method defined on a deforming space. In this case, the complexity of the geometry occurring from the fact that the mesh moves will be featured in the equation itself. In order to do so, we define a mesh velocity $\tilde{\textbf{V}}\in C(0,T;W^{3,\infty}_C(\mathcal{Q}^T))^2$ s.t.

\begin{align}
    & \tilde{\textbf{V}}\cdot \textbf{n}=0 \mbox{ on } [0,T]\times \Gamma   \label{ass:B2}\\
    & -\frac{1}{2}  \nabla \cdot (\textbf{V}(t,x)-\tilde{\textbf{V}}(t,x)) = \beta(t) > 0 \label{ass:B4}\\
    & \forall t\in [0,T], \forall x \in \Omega \;|\nabla \cdot (\textbf{V}(t,x)-\tilde{\textbf{V}}(t,x))|\leq c_* \beta(t) \label{ass:B5}
\end{align}

\noindent for a positive constant $c_*$ independent of space and time. 
\eqref{ass:B4} and \eqref{ass:B5} are there to ensure the coercivity and continuity of the method but they are not restrictive, we will explain this later in remark \ref{rq:coercivity}. 
We also decide that $\tilde{\textbf{V}}$ vanishes on the boundary between $\Gamma_N$ and $\Gamma_D$. Assumption \eqref{ass:B2} is necessary to ensure that the computational domain does not change. $\beta (t)$ will be replaced in what follows by $\beta$ and we omit the time-dependence. The condition $\tilde{\textbf{V}} \in W^{3,\infty}_C(\mathcal{Q}^T)^2$ will be used in the \textit{a priori} and \textit{a posteriori} estimates.

\noindent Note that $\tilde{\textbf{V}}$ is not identical to the flow $\textbf{V}$, but we choose it to be more regular with the above properties. Therefore $\textbf{V}-\tilde{\textbf{V}}$ is a non-zero quantity, it is the small-scale deviation from the mean (regular) flow $\tilde{\textbf{V}}$.

\noindent We introduce the flow map to distinguish a Lagrangian (or reference) variable $X$ and an Eulerian (or spatial) variable $x$. $X$ lives in a space $\hat{\Omega}$ that is later defined (and $x$ lives in $\Omega$). We carry out all the computations in $\hat{\Omega}$ and express them back in $\Omega$.

\noindent Given $\tilde{\textbf{V}}$, the associated flow map, $x = \chi(t,X)$, satisfies

\begin{equation}
\Dot{x}(t,X)=\tilde{\textbf{V}}(t,x(t,X)), \hspace{3mm} x(0,X)=X.
\end{equation}

\noindent As $\tilde{\textbf{V}}\in W^{3,\infty}_C(\mathcal{Q}^T)^2$, then the trajectories aren't crossing and we have $\chi(t,.):\hat{\Omega} \rightarrow \Omega$ is a $C^2$-diffeomorphism and the Jacobian $\textbf{F}=\textbf{F}(t,X) = [\frac{\partial x_i}{\partial X_j}]_{i,j}$ satisfies

\begin{equation}
    \Dot{\textbf{F}}(t,X)=(\nabla\tilde{\textbf{V}}(t,x))\textbf{F}(t,X), \hspace{3mm} \textbf{F}(0,X)=\textbf{I}, \hspace{3mm} x=\chi(t,X).
\end{equation}

\noindent The determinant $J = \mbox{det}(\textbf{F})$ satisfies $\Dot{J}=J(\nabla\cdot \tilde{\textbf{V}})$. Finally the smoothness of $\chi$ ensures that $J>0$ and $\textbf{F}^{-1}\textbf{F}^{-T}$ is positive definite.

\noindent For a fixed domain $\hat{\Omega} = \Omega$ with Lipschitz boundary $\Gamma_r = \Gamma$ let $\Omega (t) = \chi (t,\hat{\Omega})$. The unit normal outward $\textbf{n}_r = \textbf{n}_r(X)$ to $\Gamma_r$ and the unit normal outward $\textbf{n} = \textbf{n}(t,x)$ to $\Gamma$ are related by the formula:

\begin{equation}
    \textbf{n}(t,x)=\left(\frac{\textbf{F}^{-T}\textbf{n}_r}{|\textbf{F}^{-T}\textbf{n}_r|}\right)(t,X), \hspace{3mm} X \in \partial \hat{\Omega}, \hspace{3mm} x=\chi(t,X).
\end{equation}

\noindent Because of the vanishing normal of $\tilde{\textbf{V}}$ on $\Gamma_r$ (condition \eqref{ass:B2} above), $\hat{\Omega}=\Omega(t)=\Omega$, $\chi(t,\Gamma_D)=\Gamma_D $ and $\chi(t,\Gamma_N)=\Gamma_N $.

\noindent For any function $v(t,x)$ we introduce the $\hat{ }\,$-notation s.t. $\hat{v}(t,X) = v(t,x(t,X))$ and reciprocally $v$ can be defined thanks to $\hat{v}$. This notation also stands for functions independent of time, $\hat{v}(t,X)= v(x)$ and $v(t,x)= \hat{v}(X)$. 

\noindent Then

\begin{equation}
    \hat{v}_t=v_t + \tilde{\textbf{V}} \cdot \nabla v,\; \nabla v=\textbf{F}^{-T}\nabla_X \hat{v}\mbox{ and } \Delta v = \frac{1}{J}\nabla_X \cdot (J\textbf{F}^{-1}\textbf{F}^{-T}\nabla_X\hat{v}).
\end{equation}

\noindent When writing $\hat{f}(t,X) = f(t,x(t,X)),$ $\hat{u}_D(t,X) = u_D(t,x),$ $\hat{u}_N(t,X) = u_N(t,x),$ and $\hat{u}_0(X) = u_0(x(0,X))$, the PDE in \eqref{eq:c-d} can now be written in the moving framework, using $\tilde{\textbf{V}}$ as the deformational flow:

\begin{equation}\protect\label{eq:mmform}
\frac{\partial \hat{u}}{\partial t} + (\textbf{V}-\tilde{\textbf{V}}) \cdot \textbf{F}^{-T} \nabla_X \hat{u} - \frac{\varepsilon}{J} \nabla_X \cdot \{J\textbf{F}^{-1}\textbf{F}^{-T}\nabla_X \hat{u}\} = \hat{f}
\end{equation}

\noindent Note that $\frac{\partial \hat{u}}{\partial t}$ corresponds to the material derivative (i.e. Lagrangian framework) of $u$. Therefore, the moving advection-diffusion equation in \eqref{eq:mmform} has a much smaller advection term, given by $\textbf{V}-\tilde{\textbf{V}}$.

\noindent In the following we will discretise \eqref{eq:mmform} with a DG method in space. This approach allows us to compute the approximation on the static space $\hat{\Omega}$ with a partition $\mathcal{T}_h$ defined later.

\section{The semi-discrete interior penalty discretisation formulation for the unsteady advection-diffusion equation}
\label{sec:sd}

In this section we will discuss the DG formulation of \eqref{eq:mmform}. We start with the weak form of the equations and define the semi-discrete problem in the first two subsections that leads to the definition of the discrete variational problem \eqref{weakform:discr}. Some properties of the operators used for the semi-discretisation are stated in subsection \ref{ssec:prop}, they show the well-posedness of our problem and open the way for the error estimates that will be discussed in section \ref{sec:errest}.

\subsection{Notation and weak form}

\noindent Let $\hat{\omega} \subset \hat{\Omega}$, we define the following:

\begin{equation*}
    \lVert \hat{v} \rVert^2_{H_{\hat{\omega}}(t)} = \int_{\hat{\omega}} \hat{v}^2 J(t,\cdot)dX, \hspace{3mm} |\hat{v}|^2_{U_{\hat{\omega}}(t)} = \int_{\hat{\omega}} (\nabla_X \hat{v})^T \textbf{F}(t,\cdot)^{-1} \textbf{F}(t,\cdot)^{-T}(\nabla_X \hat{v}) J(t,\cdot)dX
\end{equation*}

\begin{remark}\label{rq:normequivalence}
    If $\hat{\omega}=\hat{\Omega}$, the index $\hat{\Omega}$ is omitted. Denoting $v(t,\chi(t,X)) = \hat{v}(X)$ and $\omega(t) = \chi(t,\hat{\omega})$ then

\begin{equation*}
    \lVert v \rVert_{L^2(\omega)}=\lVert\hat{v}\rVert_{H_{\hat{\omega}}(t)} \mbox{ and } |v|_{H^1(\omega)}=|\hat{v}|_{U_{\hat{\omega}}(t)}.
\end{equation*}

\noindent The equalities show that the metrics can be constructed within the framework of the reference variable $X$: by denoting $e$ the approximation error, the bounds on $\|\hat{e}\|_{H(t)}$ and $|\hat{e}|_{U(t)}$ represent $L^2$- and $H^1$-bounds on $e$.
\end{remark}

\noindent Writing $\hat{\mathcal{Q}}^T=[0,T]\times \hat{\Omega}$, for $p\geq 1$, $k\in \mathbb{N}$ and $Z$ a Banach space of measurable functions with values in $\hat{\Omega}$, we define $L^p(0,T;Z)$ similarly as for measurable functions with values in $\Omega$ in \eqref{LpLpOmega} and for a Sobolev space $Z(\hat{\Omega})$ we will have $Z_p(\hat{\mathcal{Q}}^T)=L^p(0,T;Z(\hat{\Omega}))$ and $Z_{C^k}(\hat{\mathcal{Q}}^T)=C^k(0,T;Z(\hat{\Omega}))$. $k$ is omitted if equal to 0.

\noindent Set $H^1_D(\hat{\Omega}) \! = \! \{\hat{v} \in H^1(\hat{\Omega}) \colon \hat{v}=0 \mbox{ on } \Gamma_D\}$.

\noindent For $\hat{u},\hat{v} \in H^1(\hat{\Omega})$ we define the following:

\begin{align*}
& a(\hat{u},\hat{v})  = \int_{\hat{\Omega}} J[\varepsilon (\textbf{F}^{-T}\nabla_X\hat{u})\cdot(\textbf{F}^{-T}\nabla_X\hat{v}) +(\textbf{V}-\tilde{\textbf{V}})\cdot(\textbf{F}^{-T}\nabla_X \hat{u})\hat{v}]\\
& l(\hat{v})   =  \int_{\hat{\Omega}} J \hat{f} \hat{v} + \int_{\Gamma_N} J \hat{u}_N \hat{v} |\textbf{F}^{-T}\textbf{n}_r|
\end{align*}

\noindent With these definitions the weak formulation of \eqref{eq:mmform} becomes:

\begin{multline} \label{weakform:ctn}
    \mbox{Find }\hat{u} \in H^1_{D\,2}(\hat{\mathcal{Q}}^T) \cap L^\infty_\infty(\hat{\mathcal{Q}}^T) \mbox{ s.t. } \hat{u}(0,\cdot)=u_0\\ 
     \mbox{and } \forall \hat{v} \in H_D^1(\hat{\Omega}),  \mbox{ for almost every } t \in [0,T],\; \int_{\hat{\Omega}} J\frac{\partial \hat{u}}{\partial t}\hat{v} = l(\hat{v}) - a(\hat{u},\hat{v})
\end{multline}

\noindent Since $J\neq 0$, $\textbf{F}^{-1}\textbf{F}^{-T}$ is positive definite and $\textbf{F}\in L^\infty_C(\mathcal{Q}^T)$, this problem has a unique solution. This solution satisfies $\frac{\partial \hat{u}}{\partial t} \in L^2_2(\hat{\mathcal{Q}}^T)$.

\noindent Integration by parts then yields

\begin{multline*}
     a(\hat{u},\hat{v}) = \int_{\hat{\Omega}} J[\varepsilon (\textbf{F}^{-T}\nabla_X\hat{u})\cdot(\textbf{F}^{-T}\nabla_X\hat{v}) -\hat{u}(\textbf{V}-\tilde{\textbf{V}})\cdot(\textbf{F}^{-T}\nabla_X \hat{v})- \nabla\cdot(\textbf{V}-\tilde{\textbf{V}}) \hat{u}\hat{v}] \\
      + \int_{\Gamma_N} J(\textbf{V}-\tilde{\textbf{V}})\cdot(\textbf{F}^{-T}\textbf{n}_r) \hat{u}\hat{v} 
\end{multline*}

The formulation \eqref{weakform:ctn} is a \textit{nonconservative} weak ALE formulation for problem \eqref{eq:c-d}. Regularity of the ALE map $\chi$ guarantees the parabolic nature of the resulting equation and consequently the existence of a unique solution to \eqref{weakform:ctn}.

\subsection{Bilinear forms and function spaces for the semi-discretisation}

To discretise \eqref{weakform:ctn}, we consider a regular and shape-regular mesh $\mathcal{T}_h = \{K\} $ that partition the computational domain $\hat{\Omega}$ into open triangles. We consider this mesh to be static. For simplicity we assume that $\hat{\Omega}$ is a polyhedron that is covered exactly by $\mathcal{T}_h$. We assume the following:

(i) The elements of $\mathcal{T}_h$ satisfy the minimal angle condition. Specifically, there is a constant $\xi_0>0$ such that $\frac{h_K}{r_K}<\xi_0$ where $h_K$ and $r_K$ denote, respectively, the diameters of the circumscribed and inscribed balls to $K$.

(ii) $\mathcal{T}_h$ is locally quasi-uniform; that is, if two elements $K_1$ and $K_2$ are adjacent in the sense that $\mu_{d-1}(\partial K_1 \cap \partial K_2)>0$ then $\mbox{diam}(K_1) \approx \mbox{diam}(K_2)$.

\noindent Here $\mu_{d-1}$ denotes the $(d-1)$-dimensional Lebesgue measure. Let $h_E$ be the size of an edge $E$. We work with a large mesh Peclet number, which means $h_E\gg \varepsilon$.

\noindent Given the discontinuous nature of the piecewise polynomial functions, let the sets of edges:

\begin{align*}
    &\mathcal{E}^{\mbox{\tiny int}}\! =\! \{ e\!=\!\partial\! K_1\! \cap\! \partial\! K_2 \colon \mu_{d-1}(\partial\! K_1 \!\cap\! \partial\! K_2)\!>\!0\}, \\
    &\mathcal{E}^{\mbox{\tiny ext}}\! =\! \{ e\!=\!\partial\! K\! \cap\! \Gamma \colon \mu_{d-1}(\partial\! K \!\cap\! \Gamma)>0\},\\
    &\mathcal{E}_N^{\mbox{\tiny ext}} =\! \{ e\!=\!\partial\! K\! \cap\! \Gamma_N \colon \mu_{d-1}(\partial\! K\! \cap\! \Gamma_N)\!>\!0\}, \\ &\mathcal{E}_D^{\mbox{\tiny ext}}\! =\! \{ e\!=\!\partial\! K\! \cap\! \Gamma_D \colon \mu_{d-1}(\!\partial\! K \!\cap\! \Gamma_D)\!>\!0\},\\
    &\mathcal{E} = \mathcal{E}^{\mbox{\tiny ext}}(\mathcal{T}_h) \cup \mathcal{E}^{\mbox{\tiny int}}(\mathcal{T}_h),\\
    &\mathcal{E}_N = \mathcal{E}_N^{\mbox{\tiny ext}}(\mathcal{T}_h) \cup \mathcal{E}^{\mbox{\tiny int}}(\mathcal{T}_h)\\
    &\mathcal{E}_D = \mathcal{E}_D^{\mbox{\tiny ext}}(\mathcal{T}_h) \cup \mathcal{E}^{\mbox{\tiny int}}(\mathcal{T}_h).
\end{align*}

\noindent We write $\textbf{n}_K$ the outward unit normal vector on the boundary $\partial K$ of an element $K$ and for a set of points $Y$ the set of elements $K$ containing $Y$: $\omega_Y = \underset{K \cap Y \neq \emptyset}{\cup} K$.

\noindent Let the inflow and outflow boundaries of $\partial \Omega$

\begin{center}
    $\Gamma_{\mbox{\tiny in}}^t = \{x \in \Gamma \colon \textbf{V}\cdot \textbf{n} < 0\}$, and $\Gamma_{\mbox{\tiny out}}^t = \{x \in \Gamma \colon \textbf{V}\cdot \textbf{n} \geq 0\}$ 
\end{center}

\noindent \eqref{ass:B2} implies that on the boundary $|\textbf{F}^{-T}\textbf{n}_r|\textbf{V}\cdot \textbf{n}=(\textbf{V}-\Tilde{\textbf{V}})\cdot \textbf{F}^{-T}\textbf{n}_r$.

\noindent Similarly, the inflow and outflow boundaries of an element K are defined by

\begin{center}
    $\partial K_{\mbox{\tiny in}}^t = \{x \in \partial K \colon (V-\Tilde{V})\cdot F^{-T}\textbf{n}_K < 0\}$, $\partial K_{\mbox{\tiny out}}^t = \{x \in \partial K \colon (V-\Tilde{V})\cdot F^{-T} \textbf{n}_K \geq 0\}$.
\end{center}

\noindent The broken Sobolev spaces and the space of piecewise polynomials associated with $\mathcal{T}_h$

\begin{equation*}
    H^k\!(\mathcal{T}_h) \! = \! \{\varphi \! \in \! L^2\! (\hat{\Omega}) \! \colon \! \forall K \! \in \! \mathcal{T}_h \; \varphi|_{K} \! \in \! H^k \! (K)\}, \; V_h \! = \! \{\varphi \! \in \! H^1 \! (\mathcal{T}_h) \! \colon \! \forall K \! \in \! \mathcal{T}_h \; \varphi|_{K} \! \in \! \mathcal{S}_p\! (K)\}
\end{equation*}

\noindent with $\mathcal{S}_p$ the space of polynomials of degree $\leq p$. We suppose $p \geq 1$.

\noindent Finally $U_h = V_h + H_D^1(\hat{\Omega})$ and $V_h^c= V_h \cap H_D^1(\hat{\Omega})$.

The jumps and averages of functions in $H^{1/2}(\mathcal{T}_h)$ are defined as follows. Let the edge $E\in \mathcal{E}^{\mbox{\tiny int}}$ be shared by two neighboring elements $K$ and $K^e$. For ${\hat{v} \in H^{1/2}(\mathcal{T}_h)}$, $\hat{v}|_{E}$ its trace on $E$ taken from inside $K$, and $\hat{v}^e|_{E}$ the one taken inside $K^e$. The jump of $\hat{v}$ and the average of a vector field $\textbf{q} \in H^{1/2}(\mathcal{T}_h)^d$ across the edge $E$ are then defined

\begin{equation*}
    [\![ \hat{v} ]\!] \! =\! \left \{ \begin{array}{ll} \hat{v}|_{E}\textbf{n}_{K} \! + \! \hat{v}^e|_{E}\textbf{n}_{K^e} & \mbox{for } E \! \in \! \mathcal{E}^{\mbox{\tiny int}} \\
    \hat{v}|_{E}\textbf{n}_{r} & \mbox{for } E \! \in \! \mathcal{E}^{\mbox{\tiny ext}}
    \end{array}
    \right.
    \mbox{ and }
    \{\!\{\textbf{q}\}\!\} \! = \! \left \{ \begin{array}{ll} \frac{1}{2} (\textbf{q}|_{E} \! + \! \textbf{q}^e|_{E}) & \mbox{for } E \! \in \! \mathcal{E}^{\mbox{\tiny int}} \\
    \textbf{q}|_{E} & \mbox{for } E \! \in \! \mathcal{E}^{\mbox{\tiny ext}}
    \end{array}
    \right.
\end{equation*}

\begin{remark}\label{rq:jump}
The definition of the jump holds for $u_D=0$, when we consider the general case we modify the definition of the jump: $[\![ \hat{v} ]\!] \! =\!(\hat{v}|_{E}-\hat{u}_D)\textbf{n}_{r}  \mbox{ for } E \! \in \! \mathcal{E}_D^{\mbox{\tiny ext}}$.
\end{remark}

Additionally we define $a_h( \cdot ,  \cdot )$ as

\begin{multline}\label{bilin:Ah}
    a_h(\hat{u},\hat{v})  = \underset{K \in \mathcal{T}_h}{\sum} \int_K J[\varepsilon (\textbf{F}^{-T}\nabla_X \hat{u})\cdot(\textbf{F}^{-T}\nabla_X \hat{v}) + (\textbf{V}-\tilde{\textbf{V}}) \cdot (\textbf{F}^{-T}\nabla_X \hat{u})\hat{v}]  \\
    - \underset{E\in \mathcal{E}_D}{\sum} \int_E J^\frac{1}{2}\varepsilon(\{\!\{\Pi_{L^2}(J^\frac{1}{2}\textbf{F}^{-T}\nabla_X\hat{u})\}\!\} \cdot \textbf{F}^{-T}[\![\hat{v}]\!] + \theta  \{\!\{\Pi_{L^2}(J^\frac{1}{2}\textbf{F}^{-T}\nabla_X\hat{v})\}\!\} \cdot \textbf{F}^{-T}[\![\hat{u}]\!] ) \\
    +  \frac{\alpha \varepsilon}{h_E} \int_E J \textbf{F}^{-T}[\![\hat{u}]\!] \cdot \textbf{F}^{-T}[\![\hat{v}]\!] - \underset{K \in \mathcal{T}_h}{\sum}\int_{\partial K_{\mbox{\tiny in}}^t} J(\textbf{V}-\tilde{\textbf{V}})\cdot (\textbf{F}^{-T}[\![\hat{u}]\!])\hat{v} 
\end{multline}

\noindent$\alpha>0$ is the interior penalty parameter that increases with the polynomial degree, let $\theta \in \{-1,1\}$, the method is called symmetric interior penalty (SIPG) when $\theta = 1$ and nonsymmetric interior penalty (NIPG) when $\theta = -1$. $\Pi_{L^2}$ denotes the elementwise orthogonal $L^2$-projection onto the finite element space $(V_h)^2$, it has the property

\begin{equation} \label{ppt:L2}
    \forall \hat{g}\in (L^2(\hat{\Omega}))^2, \; \Vert \Pi_{L^2} \hat{g} \Vert_{L^2(\hat{\Omega})^2} \leq \Vert \hat{g} \Vert_{L^2(\hat{\Omega})^2} 
\end{equation}

\noindent The DG method is based on an upwind discretisation for the convective term and a (non-)symmetric interior penalty discretisation for the Laplacian, it is formulated as

\begin{multline}\label{weakform:discr}
\mbox{Find } \hat{u}_h \! \in \! C^1\! (0,T;V_h) \mbox{ s.t. } \forall t \! \in \! [0,T], \forall \hat{v}_h \! \in \! V_h \; \int_{\hat{\Omega}} \! J\frac{\partial \hat{u}_h}{\partial t}\hat{v}_h\! =\! l_h(\hat{v}_h)\! -\! a_h(\hat{u}_h,\! \hat{v}_h) \\
\mbox{where } \hat{u}_h(0, \cdot ) \in V_h \mbox{ is a projection of } \hat{u}_0( \cdot ) \mbox{ onto $V_h$.} 
\end{multline}

\begin{multline*}
\mbox{With } l_h(\hat{v})  = \underset{K \in \mathcal{T}_h}{\sum} \int_{K} J \hat{f} \hat{v} +  \underset{E \in \mathcal{E}_N^{\mbox{\tiny ext}}}{\sum}\int_{E} J \hat{u}_N \hat{v} |\textbf{F}^{-T}\textbf{n}_r|.\\
\end{multline*}

Finally, for $\hat{u},\hat{v} \in H^1(\mathcal{T}_h)$

\begin{align*}
    d_h(\hat{u},\hat{v}) = & \! \underset{K \in \mathcal{T}_h}{\sum} \int_K J\varepsilon (\textbf{F}^{-T}\nabla_X \hat{u})\cdot(\textbf{F}^{-T}\nabla_X \hat{v}) - \hat{u} \hat{v} \nabla_x \cdot (\textbf{V}-\Tilde{\textbf{V}}) \\
    f_h(\hat{u},\hat{v}) = & - \underset{K \in \mathcal{T}_h}{\sum} \int_K J(\textbf{V}-\tilde{\textbf{V}}) \cdot (\textbf{F}^{-T}\nabla_X \hat{v})\hat{u}  + \underset{K \in \mathcal{T}_h}{\sum}\int_{\partial K_{\mbox{\tiny out}}^t} J(\textbf{V}-\tilde{\textbf{V}}) \cdot \textbf{F}^{-T} [\![ \hat{v} ]\!] \hat{u} \nonumber\\
    j_h(\hat{u},\hat{v}) = & \! \underset{E\in \mathcal{E}_D}{\sum} \int_E \frac{\alpha}{h_E} J \textbf{F}^{-T}[\![\hat{u}]\!] \cdot \textbf{F}^{-T}[\![\hat{v}]\!] \nonumber\\
    \tilde{a}_h(\hat{u},\hat{v}) = & d_h(\hat{u},\hat{v})+f_h(\hat{u},\hat{v})+  \varepsilon j_h(\hat{u},\hat{v}) \nonumber\\
    \tilde{p}_h(\hat{u},\hat{v}) = &  \! \underset{E\in \mathcal{E}_D}{\sum} \! \int_E \! J^\frac{1}{2} \! \varepsilon ( \! \{\!\{\Pi_{L^2} ( J^\frac{1}{2} \! \textbf{F}^{-T} \nabla\! _X \hat{u}) \}\!\} \cdot  \textbf{F}^{-T} \! [\![ \hat{v} ]\!] \! + \! \theta  \! \{\!\{ \Pi_{L^2} ( J^\frac{1}{2} \textbf{F}^{-T} \! \nabla\! _X \hat{v}) \}\!\} \cdot \textbf{F}^{-T} \! [\![ \hat{u} ]\!]) \nonumber\\
\end{align*}

Integration by parts then gives:

\begin{equation}\label{bilin:Ahint}
    a_h(\hat{u},\hat{v}) = \tilde{a}_h(\hat{u},\hat{v})-\tilde{p}_h(\hat{u},\hat{v}).
\end{equation}

And we notice that 

\begin{equation}\label{bilin:AhAhtilde}
    \forall \hat{v} \in H^1_D(\hat{\Omega}) \; a_h(\hat{v},\hat{v}) = \tilde{a}_h(\hat{v},\hat{v}).
\end{equation}

\subsection{Properties of the bilinear operator}\label{ssec:prop}

\noindent In this subsection we state some properties of the operators. The coercivity lemma proves the stability of the method. Convergence cannot be shown directly since the operator is inconsistent, but the inconsistency lemma will be useful in the proof of the \textit{a priori} error theorem \ref{thm:apriori}. The convergence of the method is then subject to the conditions described in remark \ref{rq:velocity}. The continuity and inf-sup condition lemmas show well-posedness and will be useful for the development of \textit{a posteriori} error estimates. Before the development of these properties, we will recall some classical tools for studying finite element discretisations as well as the norms and semi-norms used for analysing the semi-discretisation.

\noindent By (3.41) in \cite{feistauer07} there exist the following inverse and trace inequalities that are independent of time.

\begin{align}
    \lVert \hat{v}_h\rVert _{\partial K}^2 & \leq C_T h_K^{-1} \lVert \hat{v}_h \rVert _K^2 \label{ppt:ite}\\
    \lVert \nabla_X \hat{v}_h \rVert _{K} & \leq C_I h_K^{-1} \lVert \hat{v}_h \rVert _K \label{ppt:ie}
\end{align}

\noindent $C_T>0$ and $C_I>0$ depend on the polynomial degree.

\noindent To analyse the spatial error in \eqref{norm:A} and in the criteria \eqref{west}, let $\hat{g}_q^{\hat{\omega}}=|\hat{\omega}|^{-\frac{1}{q}} \lVert \hat{g} \rVert_{L^q(\hat{\omega})}$ for $\hat{\omega} \subset \hat{\Omega}$, $q\in [1,\infty]$ and $\hat{g}\in L^p(\hat{\Omega})$ with $|\hat{\omega}|$ the Lebesgue measure of $\hat{\omega}$. That is the $L^q$-norm of $g$ normalized by the $L^q$-norm of the indicator function of $\hat{\omega}$. And we write for all $(t,X)$

\begin{equation}\label{elts}
    \delta \! = \! (\textbf{V}-\Tilde{\textbf{V}})^2, \, M \! =\! \frac{1}{J} \mbox{ and } a \mbox{ the pointwise maximal eigenvalue of } J\textbf{F}^{-1}\textbf{F}^{-T}
\end{equation}

\noindent In what follows we will often use that since $\det (J\textbf{F}^{-1}\textbf{F}^{-T})=1$ and $d=2$, the maximal eigenvalue of  $J\textbf{F}^{-1}\textbf{F}^{-T}$ is the same as the maximal eigenvalue of $(J\textbf{F}^{-1}\textbf{F}^{-T})^{-1}$.

\noindent We define the norms and semi-norm

\begin{align}
    \forall \hat{v}\in H^1(\mathcal{T}_h), \;  |||\hat{v} |||^2 &= \underset{K \in \mathcal{T}_h}{\sum}[  \varepsilon |\hat{v}|_{U_K}^2 + \beta \lVert \hat{v}\rVert ^2_{H_K}] +  \varepsilon j_h(\hat{v},\hat{v}) \label{norm:nrj}\\
    \forall \textbf{q}\in H^0(\hat{\Omega})^2,\; |\textbf{q}|_{*} &= \underset{\hat{v}\in H^1_D(\hat{\Omega})-\{0\}}{\mbox{sup}} \frac{1}{ |||\hat{v} |||}\int_{\hat{\Omega}} J\textbf{q}\cdot \textbf{F}^{-T}\nabla_X \hat{v} \label{norm:helm}\\
    \forall \hat{v} \in H^1(\mathcal{T}_h), \; |\hat{v}|^2_{A} &= | (\textbf{V}-\Tilde{\textbf{V}}) \hat{v}|^2_{*} + \underset{E\in \mathcal{E}_D}{\sum}(\beta + \frac{\delta_\infty^{\omega_E}}{\varepsilon}) h_E J_1^{\omega_E} \int_E [\![ \hat{v} ]\!]^2 \label{norm:A}
\end{align}

\noindent The first norm is the DG-norm associated with the discretisation of \eqref{eq:mmform}. The seminorm $|\cdot|_{*}$ is linked to the Helmholtz decomposition, in particular $\textbf{q}$ is called divergence-free when $|\textbf{q}|_{*} =0$ holds. The third norm measures the error of the transport behaviour. Since $J>0$, these norms and seminorms are well defined for all $t$.


\noindent Finally we define the notation $\lesssim$ as a bound that is valid up to a constant that is independent of the local mesh size $h$, diffusion coefficient $\varepsilon$, speed $\delta$ and time $t$

\begin{lemma}\label{lem:coercivity}
(Coercivity) For $\alpha$ large enough (depending on the value of the scalar $C_T$ in the inverse trace inequality \eqref{ppt:ite}) $a_h$ is $ |||\cdot |||$-coercive. For all $\hat{v}_h \in U_h$

\begin{equation}
     |||\hat{v}_h |||^2 \lesssim a_h(\hat{v}_h,\hat{v}_h) 
\end{equation}

\end{lemma}

\begin{proof}
Let $\hat{v}_h \in V_h$, by noting that $a_h = \frac{1}{2} ( \mbox{ \eqref{bilin:Ah} } + \mbox{ \eqref{bilin:Ahint} } )$:

\begin{equation*}
    a_h(\hat{v}_h,\hat{v}_h) \geq  \frac{\beta}{2} \underset{K \in \mathcal{T}_h}{\sum}  ||\hat{v}_h ||_{H_K}^2 + \varepsilon (j_h(\hat{v} _h,\hat{v}_h) + \underset{K \in \mathcal{T}_h}{\sum} |\hat{v}_h|_{U_K}^2) - \Tilde{p}_h(\hat{v}_h,\hat{v}_h)
\end{equation*}

\noindent If $\theta = -1$, then $\Tilde{p}_h(\hat{v}_h,\hat{v}_h)=0$, thus $a_h(\hat{v}_h,\hat{v}_h)  \geq \frac{1}{2}  |||\hat{v}_h |||^2$. 

\noindent When $\theta=1$, using successively Young's inequality, \eqref{ppt:ite}, \eqref{ppt:L2} and \eqref{ppt:ie}, we obtain

\begin{equation*}
2\underset{E\in \mathcal{E}_D}{\sum} \int_E \{\!\{\Pi_{L^2} ( J^\frac{1}{2}\textbf{F}^{-T} \nabla \hat{v}_h) \}\!\} \cdot J^\frac{1}{2}\textbf{F}^{-T}[\![ \hat{v}_h ]\!]  \leq \sqrt{\frac{C_T}{\alpha}}(j_h(\hat{v} _h,\hat{v}_h) + \underset{K \in \mathcal{T}_h}{\sum} |\hat{v}_h|_{U_K}^2).
\end{equation*}


\begin{align*}
& \mbox{Thus } |||\hat{v}_h |||^2 \leq \max {(2,(1-\sqrt{\frac{C_T}{\alpha}})^{-1})}  a_h(\hat{v}_h,\hat{v}_h), \mbox{ that is bounded for } \alpha >C_T.
\end{align*}
\end{proof}

As often used in the literature, we will suppose that $\alpha \geq 2C_T$.

\begin{lemma} \label{lem:inconsistency}
(Inconsistency) Let $\hat{u}$ be a solution of \eqref{weakform:ctn}, let $t\in [0,T]$ such that \eqref{eq:mmform} is resolved analytically on $\hat{\Omega}$ by $\hat{u}(t,\cdot)$, let $\hat{u}_h$ solve \eqref{weakform:discr} and $\hat{e}=\hat{u}-\hat{u}_h$. Then $\forall \hat{v}_h \in V_h$

\begin{equation*}
    \int_{\hat{\Omega}} J\frac{\partial \hat{e}}{\partial t}\hat{v}_h+a_h(\hat{e},\hat{v}_h)\! =\! \varepsilon \underset{E\in \mathcal{E}_D}{\sum} \int_E \{\!\{ J^\frac{1}{2}\textbf{F}^{-T} \nabla_X \hat{u} - \Pi_{L^2} ( J^\frac{1}{2}\textbf{F}^{-T} \nabla_X \hat{u}) \}\!\} \cdot J^\frac{1}{2}\textbf{F}^{-T} [\![ \hat{v}_h]\!]
\end{equation*}

\end{lemma}

\begin{proof}

Let $ \hat{v}_h \in V_h $, with \mbox{ \eqref{bilin:Ah} }, $\hat{u}(t,\cdot ) \in H^1(\hat{\Omega})$ and integration by parts yields:

\begin{align*}
     \int_{\hat{\Omega}} & J\frac{\partial \hat{u}}{\partial t}\hat{v}_h \! + \! a_h(\hat{u},\hat{v}_h) 
    \!= \! \underset{K \in \mathcal{T}_h}{\sum} \int_K (J(\textbf{V}-\Tilde{\textbf{V}}) \cdot  \textbf{F}^{-T} \nabla_X \hat{u} - \varepsilon \nabla_X \cdot \{ J\textbf{F}^{-1}\textbf{F}^{-T} \nabla_X \hat{u}\}) \hat{v}_h \\
    &+ \varepsilon \int_{\partial K} J\textbf{F}^{-1}\textbf{F}^{-T} \nabla_X \hat{u} \cdot \textbf{n}_r \hat{v}_h - \varepsilon \underset{E\in \mathcal{E}_D}{\sum} \int_E \{\!\{\Pi_{L^2} ( J^\frac{1}{2}\textbf{F}^{-T} \nabla_X \hat{u}) \}\!\} \cdot J^\frac{1}{2}\textbf{F}^{-T} [\![ \hat{v}_h ]\!] \\
    = \! & \int_{\hat{\Omega}}\! J\frac{\partial \hat{u}_h}{\partial t}\hat{v}_h\!+\!a_h(\hat{u}_h,\!\hat{v}_h)\!+\!  \varepsilon\! \underset{E\in \mathcal{E}_D}{\sum} \int_E \{\!\{ \! J^\frac{1}{2} \! \textbf{F}^{-T} \! \nabla_X \hat{u}\! -\! \Pi_{L^2} \! ( J^\frac{1}{2} \! \textbf{F}^{-T} \! \nabla_X \hat{u})\! \}\!\}\! \cdot\! J^\frac{1}{2}\! \textbf{F}^{-T} [\![ \hat{v}_h]\!]
\end{align*}

\noindent The last equality follows from \eqref{eq:mmform} and \eqref{weakform:discr}.
\end{proof}

This lemma establishes an inconsistency since the right hand-side is in general non-zero. This deviation from the continuous problem is the result of the projection of the fluxes of the interior penalty formulation that is necessary for the stability of the method. This deviation goes to $0$ when $h$ is small or $p$ is large.

\begin{lemma} \label{lem:continuity}
(Continuity) For  $\hat{v} \in H^1_D(\hat{\Omega})$, $\hat{w}_1 ,\hat{w}_2 \in U_h$, the following properties hold:

\begin{align}
    \vert d_h(\hat{w}_1 ,\hat{w}_2)\vert  & \leq  \max (1,c_*)|||\hat{w}_1 |||\cdot  |||\hat{w}_2 |||, \label{ineq:dh}\\
    \vert j_h(\hat{w}_1 ,\hat{w}_2 )\vert  & \leq  |||\hat{w}_1 |||\cdot  |||\hat{w}_2 |||,\label{ineq:jh} \\
     \vert f_h(\hat{w}_1 ,\hat{v})\vert  & \leq \vert (\textbf{V}-\Tilde{\textbf{V}})\hat{w}_1\vert _*\cdot  |||\hat{v} |||. \label{ineq:fh}
\end{align}

\end{lemma}

\begin{proof}

\eqref{ineq:dh} and \eqref{ineq:jh} follow from the Cauchy-Schwarz inequality and \eqref{ass:B5}.

\begin{align*}
    \vert f_h(\hat{w}_1 ,\hat{v})\vert  &\leq \vert \underset{K \in \mathcal{T}_h}{\sum} \int_K \hat{w}_1J(\textbf{V}-\Tilde{\textbf{V}})\cdot \textbf{F}^{-T}\nabla_X \hat{v}\vert  \leq  |||\hat{v} |||\frac{\vert \int_{\hat{\Omega}} \hat{w}_1J(\textbf{V}-\Tilde{\textbf{V}})\cdot \textbf{F}^{-T}\nabla_X \hat{v}\vert }{ |||\hat{v} |||} \\
    &\leq \vert (\textbf{V}-\Tilde{\textbf{V}})\hat{w}_1\vert _*\cdot  |||\hat{v} |||
\end{align*}

\end{proof}

\begin{lemma} \label{lem:infsup} (Inf-Sup) $\tilde{a}_h$ satisfies the following condition:

\begin{equation*}
    \underset{\hat{v} \in H_D^1(\hat{\Omega})-\{0\}}{\inf} \underset{\hat{w}\in H_D^1(\hat{\Omega})-\{0\}}{\sup} \frac{\tilde{a}_h(\hat{v},\hat{w})}{( |||\hat{v} |||+|(\textbf{V}-\Tilde{\textbf{V}})\hat{w}|_*)\cdot  |||\hat{w} |||} \geq C > 0
\end{equation*}

\end{lemma}

\begin{proof}
Let $\hat{v} \in H^1_D(\hat{\Omega})-\{0\}$ and $ \sigma \in ]0;1[ $. Then by definition of the norm $|\cdot|_{*}$ \eqref{norm:helm} there exists $\hat{w}_\sigma \in H^1_D(\hat{\Omega})$ s.t. $ |||\hat{w}_\sigma ||| = 1$ and

\begin{equation*}
     f_h(\hat{v},\hat{w}_\sigma) =  - \int _{\hat{\Omega}} J\hat{v}(\textbf{V}-\Tilde{\textbf{V}})\cdot \textbf{F}^{-T} \nabla_X \hat{w}_\sigma \geq \sigma |(\textbf{V}-\Tilde{\textbf{V}})\hat{v}|_*.
\end{equation*}

\noindent Then $\tilde{a}_h(\hat{v},\hat{w}_\sigma) \geq d_h(\hat{v},\hat{w}_\sigma) + \sigma |(\textbf{V}-\Tilde{\textbf{V}})\hat{v}|_* \geq \sigma |(\textbf{V}-\Tilde{\textbf{V}})\hat{v}|_*-C_1 |||\hat{v} |||$ with $C_1>0.$

\noindent Let us define $\hat{v}_\sigma \! \in \! H^1_D(\hat{\Omega})$ as

\begin{equation*}
    \hat{v}_\sigma =\hat{v} + \frac{ |||\hat{v} |||}{1+C_1} \hat{w}_\sigma \mbox{ s.t. }  |||\hat{v}_\sigma ||| \leq  |||\hat{v} ||| (1+\frac{1}{1 + C_1})
\end{equation*}

\noindent Thus by \eqref{bilin:AhAhtilde} and \eqref{ineq:dh}, \eqref{ineq:jh}, \eqref{ineq:fh}
\begin{equation*}
    \tilde{a}_h(\hat{v},\hat{v}_\sigma) \! = \!  a(\hat{v},\hat{v}) \! + \!  |||\hat{v} ||| \frac{a(\hat{v},\hat{w}_\sigma)}{1+C_1} \geq  |||\hat{v} ||| \cdot [ |||\hat{v} |||+\frac{\sigma |(\textbf{V}-\Tilde{\textbf{V}})\hat{v}|_*  - C_1 |||\hat{v} |||}{1+C_1}].
\end{equation*}

\noindent Finally
\begin{equation*}
    \underset{w\in H^1_D(\Omega)-\{0\}}{\sup} \frac{\tilde{a}_h(\hat{v},\hat{w})}{ |||\hat{w} |||}\geq \frac{\tilde{a}_h(\hat{v},\hat{v}_\sigma)}{ |||\hat{v}_\sigma |||} \geq \frac{\sigma |(\textbf{V}-\Tilde{\textbf{V}})\hat{v}|_*+ |||\hat{v} |||}{2C_1}.
\end{equation*}

\end{proof}

\noindent $C_1 = O(\max (1,c_*))$ with $c_*$ defined in assumption \eqref{ass:B5}. Since the norms are weighted, the inf-sup constant $C$ does not depend on time, the norms do.

\begin{remark}\label{rq:coercivity}
Since we have all the tools, we can explain why \eqref{ass:B4} and \eqref{ass:B5} are not restrictive. When looking for solutions $\hat{u}=e^{\gamma_0t}\hat{w}$, \eqref{eq:mmform} becomes
\begin{equation}
\frac{\partial \hat{w}}{\partial t} + (\textbf{V}-\tilde{\textbf{V}}) \cdot \textbf{F}^{-T} \nabla_X \hat{w} - \frac{\varepsilon}{J} \nabla_X \cdot \{J\textbf{F}^{-1}\textbf{F}^{-T}\nabla_X \hat{w}\} + \gamma_0 \hat{w}= \hat{f}e^{-\gamma_0 t}
\end{equation}

And the method holds by adding a reaction term $\int_K \gamma_0\hat{v}_1 \hat{v}_2$ to $d_h(\hat{v}_1,\hat{v}_2)$ and change \eqref{ass:B4} and \eqref{ass:B5} into

\begin{align*}
    & \gamma_0 -\frac{1}{2}  \nabla \cdot (\textbf{V}(t,x)-\tilde{\textbf{V}}(t,x)) = \beta(t) > 0 \\
    & \forall t\in [0,T], \forall x \in \Omega \;|\gamma_0-\nabla \cdot (\textbf{V}(t,x)-\tilde{\textbf{V}}(t,x))|\leq c_* \beta(t)
\end{align*}

Which does not change any of the properties stated in the previous section.
\end{remark}

\section{Error estimates}\label{sec:errest}

This section dedicates two major subsections to the study of error estimates. The former \textit{a priori} estimate first outlines the convergence of the method when the mesh velocity is chosen regular enough (see remark \ref{rq:velocity}). It concludes in remark \ref{rq:apriori} on the dependence of the error on a balance between the remaining advection speed $|\textbf{V}-\Tilde{\textbf{V}}|$ and the higher order derivatives of the mesh velocity $\Tilde{\textbf{V}}$. The remaining subsections derive an \textit{a posteriori} error estimate, defined in \eqref{stestim} and shown to be local. Theorem \ref{thm:apstsp-ti} concludes on its reliability.

\subsection[An \textit{a priori} error estimate]{An \textit{a priori} error estimate for the semi-discrete formulation}

In this subsection we consider the \textit{a priori} estimate to the semi-discretisation. In theorem \ref{thm:apriori}, we highlight a balance between $ \Vert \textbf{V}-\Tilde{\textbf{V}} \Vert ^2/\varepsilon$ and $ \Vert \nabla \cdot \Tilde{\textbf{V}} \Vert $. We then give the approximation lemma \ref{lem:H2} to explicitly expose the dependence in the mesh velocity $\Tilde{\textbf{V}}$. The conclusion is given in remark \ref{rq:apriori} for a DG method with polynomial order $1\leq p \leq 2$. Additionally, unlike \cite{cangiani17}, where \textbf{Theorem 46} gives a convergence rate when $J^\frac{1}{2}\textbf{F}^{-T} \nabla_X \hat{u} \in H^l(\hat{\Omega})^2$, in lemma \ref{lem:H2} we express a sufficient condition on the ALE velocity for the semi-discrete method to be convergent. 

Since we want a convergence theorem, we want all the \textit{a priori} error estimate to be expressed in the spatial variable. To do so we will start by using the weighted norms and use remark \ref{rq:normequivalence}.

In this section we suppose higher regularity for the solution of \eqref{weakform:ctn}, namely $\hat{u} \in H^2_2(\hat{\mathcal{Q}}^T)$ and $\frac{\partial\hat{u}}{\partial t} \in H^1_2(\hat{\mathcal{Q}}^T)$, this implies also that $\hat{u}$ satisfies \eqref{eq:mmform} almost everywhere and $\hat{u} \in H^1_C(\hat{\mathcal{Q}}^T)$. 

To do so, we notice that since $\chi(t,\cdot)\in C^2(\hat{\Omega},\Omega)$, such a regularity requirement is equivalent to having $u \in H^2_2(\mathcal{Q}^T)$ and $\frac{\partial u}{\partial t} \in H^1_2(\mathcal{Q}^T)$ that satisfies \eqref{eq:c-d} almost everywhere. We therefore can use the regularity theorems from \cite{roubicek13} and \cite{evans10}.

For the analysis we need to define $\Pi_{L^2}\colon (H^0(\mathcal{T}_h))^2 \to V_h^2$ the $L^2$-projection and $P_h(t)\colon H^0(\mathcal{T}_h) \to V_h$ the weighted-projections s.t.

\begin{equation}\label{proj:weighted}
    \forall \hat{v} \in H^0(\mathcal{T}_h),\; \forall \hat{v}_h \in V_h,\; \int_K J(t,\cdot)(\hat{v}-P_h\hat{v})\hat{v}_h=0 .
\end{equation}

In the following lemma we call

\begin{align}
    &\hat{e} \! = \! (\hat{u}\! -\! P_h\hat{u})\! + \! (P_h\hat{u} \! - \! \hat{u}_h) \! =\!  \hat{e}_p \! +\!  \hat{e}_h \mbox{ and } \hat{\textbf{E}}_p \! = \! J^\frac{1}{2}\textbf{F}^{-T}\nabla_X \hat{u} \! - \! \Pi_{L^2}(J^\frac{1}{2}\textbf{F}^{-T}\nabla_X \hat{u})\label{fct:Ep}\\
    &\hat{N} \colon [0,T] \times H^1(\mathcal{T}_h) \to \mathbb{R}, \mbox{ the seminorm s.t. } \hat{N}(t,\hat{v})^2 = |\hat{v}|^2_{U(t)} + 2j_h(t;\hat{v},\hat{v})\label{norm:N}\\
    &\hat{B} \colon H^{1/2}(\mathcal{T}_h)^2 \to \mathbb{R}, \mbox{ the seminorm s.t. } \hat{B}(\hat{\textbf{q}})^2 = \frac{1}{C_T} \underset{\mathcal{E}_D}{\sum}h_E\int_E \{\!\{\hat{\textbf{q}}\}\!\}^2 \label{norm:B}
\end{align}

\noindent Notice that $\hat{\textbf{E}}_p \in H^1(\mathcal{T}_h)^2$ and define $N(t,\cdot)$ and $B(t,\cdot)$ as the corresponding time dependant seminorm s.t. $N(t,v)=\hat{N}(t,\hat{v})$ and $B(t,\textbf{q})=\hat{B}(\hat{\textbf{q}})$ and $e_p$ and $\textbf{E}_p$ are the corresponding functions in the Eulerian variable. 

\noindent Recalling that $\dot{J}= J(\nabla \cdot \tilde{\textbf{V}})$, $\dot{\textbf{F}}=(\nabla \tilde{\textbf{V}})\textbf{F}$ (implying $\dot{\textbf{F}^{-1}}=-\textbf{F}^{-1}(\nabla \tilde{\textbf{V}})$) yields

\begin{equation*}
     \frac{d}{d t} \Vert \hat{v} \Vert ^2_{U_K(t)} =   \int_K J(\textbf{F}^{-T}\nabla_X\hat{v})^T(\nabla \cdot \Tilde{\textbf{V}}-\nabla \Tilde{\textbf{V}}-(\nabla \Tilde{\textbf{V}})^T)\textbf{F}^{-T}\nabla_X\hat{v} 
\end{equation*}
 
\noindent for $C_0 (t) = \frac{1}{2} \Vert \nabla \cdot \tilde{\textbf{V}} \Vert _{L^1(0,t;L^\infty(\Omega))}$ and $C_1 (t) = C_0 (t) +  \Vert \frac{\nabla \Tilde{\textbf{V}}+(\nabla \Tilde{\textbf{V}})^T}{2} \Vert _{L^1(0,t;L^\infty(\Omega))}$. Using Gronwall's inequality we have:

\begin{equation}\label{ineq:Unorm}
     \Vert \hat{v} \Vert _{U(0)} e^{-C_1(t)}  \leq  \Vert \hat{v} \Vert _{U(t)} \leq  \Vert \hat{v} \Vert _{U(0)} e^{C_1(t)}   
\end{equation}

\noindent Similarly

\begin{align}
     \Vert \hat{v} \Vert _{H(0)} e^{-C_0(t)} & \leq  \Vert \hat{v} \Vert _{H(t)} \leq  \Vert \hat{v} \Vert _{H(0)} e^{C_0(t)}   \label{ineq:Hnorm} \\
    \int_E \hat{v}^2 e^{-2C_0(t)} & \leq \int_E J \hat{v}^2 \leq  \int_E \hat{v}^2 e^{2C_0(t)}  \label{ineq:E} \\
    j_h(0;\hat{v},\hat{v}) e^{-2C_1(t)} & \leq j_h(t;\hat{v},\hat{v}) \leq  j_h(0;\hat{v},\hat{v}) e^{2C_1(t)}  \label{ineq:jtime}
\end{align}

We can now state the \textit{a priori} error estimate that will be interpreted in remark \ref{rq:apriori} with help of lemma \ref{lem:H2}.

\begin{theorem}\label{thm:apriori} Let $T>0$ and $\hat{u}_h$ be the solution of \eqref{weakform:discr}, $\hat{e}_p$ and $\hat{E}_p$ be defined in \eqref{fct:Ep} and $N(t,\cdot)$ and $B(t,\cdot)$ be the seminorms respectively defined in \eqref{norm:N} and \eqref{norm:B}.

\noindent \textbf{For SIPG ($\theta=1$)}

\begin{multline}\label{ineq:sipg}
     \Vert e \Vert ^2_{L^2_\infty(\mathcal{Q}^T)}+\frac{\varepsilon}{4}  \Vert N(\cdot,e) \Vert ^2_{L^2([0,T])}\leq C_S( \Vert e(0) \Vert ^2_{L^2(\Omega)} + 2  \Vert e_p \Vert ^2_{L^2_\infty(\mathcal{Q}^T)}\\
    +30\varepsilon \Vert N(\cdot,e_p) \Vert ^2_{L^2([0,T])} + \frac{97}{8} \varepsilon  \Vert B(\cdot,\textbf{E}_p) \Vert ^2_{L^2([0,T])})
\end{multline}

\noindent with $C_S = \exp [\frac{8 + 12  e^{4C_0^\Omega(T)}}{\varepsilon}  \Vert \textbf{V}-\Tilde{\textbf{V}} \Vert ^2_{L^\infty_2(\mathcal{Q}^T)^2} + \frac{1}{2} \Vert \nabla \cdot \Tilde{\textbf{V}} \Vert _{L^\infty_1(\mathcal{Q}^T)}]$ holds.

\noindent \textbf{For NIPG ($\theta=-1$)}

\begin{multline}\label{ineq:nipg}
     \Vert e \Vert ^2_{L^2_\infty(\mathcal{Q}^T)}+\frac{\varepsilon}{2}  \Vert N(\cdot,e) \Vert ^2_{L^2([0,T])}\leq C_N( \Vert e(0) \Vert ^2_{L^2(\Omega)} + 2  \Vert e_p \Vert ^2_{L^2_\infty(\mathcal{Q}^T)} \\
    +4\varepsilon \Vert N(\cdot,e_p) \Vert ^2_{L^2([0,T])} + \frac{17}{8}\varepsilon  \Vert B(\cdot,\textbf{E}_p) \Vert ^2_{L^2([0,T])})
\end{multline}

\noindent with $C_N = \exp [\frac{1 + 2 e^{4C_0^\Omega(T)}}{\varepsilon}  \Vert \textbf{V}-\Tilde{\textbf{V}} \Vert ^2_{L^\infty_2(\mathcal{Q}^T)^2} + \frac{1}{2} \Vert \nabla \cdot \Tilde{\textbf{V}} \Vert _{L^\infty_1(\mathcal{Q}^T)}]$ holds.
\end{theorem}

\begin{proof}
\noindent In the beginning of the proof we work with the generic form of interior penalty methods, namely $\theta \in \{ -1, 1\}$. We have the following inequalities

\begin{align*}
    \int _{\hat{\Omega}}J\frac{\partial \hat{e}}{\partial t} \hat{e}  = \frac{1}{2} \frac{d}{d t}  \Vert e \Vert _{L^2(\Omega)}^2  -  \int_{\hat{\Omega}} \frac{\nabla\cdot \tilde{\textbf{V}}}{2} J \hat{e}^2 & \geq \frac{1}{2} \frac{d}{d t}  \Vert e \Vert _{L^2(\Omega)}^2 - \frac{ \Vert \nabla\cdot \tilde{\textbf{V}} \Vert _\infty}{2} \Vert e \Vert _{L^2(\Omega)}^2 \\
    \int _{\hat{\Omega}}J\frac{\partial \hat{e}_p}{\partial t} \hat{e}_p  = \frac{1}{2} \frac{d}{d t}  \Vert e_p \Vert _{L^2(\Omega)}^2 \! - \! \int_{\hat{\Omega}} \frac{\nabla\cdot \tilde{\textbf{V}}}{2} J \hat{e}_p^2 & \leq \frac{1}{2} \frac{d}{d t}  \Vert e_p \Vert _{L^2(\Omega)}^2 + \frac{ \Vert \nabla\cdot \tilde{\textbf{V}} \Vert _\infty}{2} \Vert e_p \Vert _{L^2(\Omega)}^2. \\
\end{align*}

\noindent and the identity 

\begin{align*}
    \int _{\hat{\Omega}}J\frac{\partial \hat{e}}{\partial t} \hat{e} + a_h(\hat{e},\hat{e}) & =  \int _{\hat{\Omega}}J\frac{\partial \hat{e}}{\partial t} \hat{e}_p + \int _{\hat{\Omega}}J\frac{\partial \hat{e}}{\partial t} \hat{e}_h + a_h(\hat{e},\hat{e}_p) + a_h(\hat{e},\hat{e}_h) \\
    & = \int _{\hat{\Omega}}J\frac{\partial \hat{e}_p}{\partial t} \hat{e}_p \! + \! \underset{0}{\underbrace{\int _{\hat{\Omega}}J\frac{\partial \hat{e}_h}{\partial t} \hat{e}_p}} \! + \! a_h(\hat{e},\hat{e}_p) \! + \! \underset{\varepsilon \underset{\mathcal{E}_D}{\sum}\int_E \{\!\{\hat{\textbf{E}}_p\}\!\}\cdot J^\frac{1}{2} \textbf{F}^{-T}[\![ \hat{e}_h ]\!]}{\underbrace{\int _{\hat{\Omega}}J\frac{\partial \hat{e}}{\partial t} \hat{e}_h \! + \!a_h(\hat{e},\hat{e}_h)}}
\end{align*}

\noindent holds. Note that since $\hat{u}\in H^2(\hat{\Omega})$, $[\![ \hat{e} ]\!]=[\![ \hat{u}_h ]\!]$.

\noindent We will also use the following bounds:

\begin{align}
    &\forall \hat{v} \in H^1(\mathcal{T}_h) \; \underset{E\in \mathcal{E}_D}{\sum} h_E \int_E \{\!\{\Pi_{L^2} ( J^\frac{1}{2}\textbf{F}^{-T} \nabla_X \hat{v}) \}\!\}^2 \leq  C_T  \Vert \hat{v}  \Vert _{U(t)}^2 \label{ineq:51}\\
    &\forall \hat{v} \in V_h \; \underset{E\in \mathcal{E}_D}{\sum} h_E \int_E J\hat{v}^2 \leq  C_T e^{4C_0^\Omega} \Vert \hat{v} \Vert _{H(t)}^2 \label{ineq:52}\\
    &\forall \hat{v},\hat{w} \in  H^1(\mathcal{T}_h)\; \forall c >0, \; j_h(\hat{v},\hat{w}) \leq  \frac{c}{2}j_h(t;\hat{v},\hat{v}) + \frac{1}{2c}j_h(t;\hat{w},\hat{w})\label{ineq:53}
\end{align}

\noindent with $C_T$ defined by \eqref{ppt:ite}. \eqref{ineq:51} follows from the trace inequality \eqref{ppt:ite} and \eqref{ppt:L2}, \eqref{ineq:52} is a consequence of \eqref{ppt:ite} and \eqref{ineq:Hnorm}, and \eqref{ineq:53} by Cauchy-Schwarz and Young's inequalities. Estimating 

\begin{multline}\label{ineq:I}
I=\frac{1}{2} (\frac{d}{d t}  \Vert e \Vert _{L^2(\Omega)}^2 - \frac{d}{d t}  \Vert e_p \Vert _{L^2(\Omega)}^2) + \varepsilon |\hat{e}|^2_{U(t)} + \varepsilon j_h(t;\hat{e},\hat{e}) \\
\leq \underset{K \in \mathcal{T}_h}{\sum} - \int_K J \hat{e} (\textbf{V}-\Tilde{\textbf{V}}) \cdot  \textbf{F}^{-T}\nabla_X \hat{e}- J \textbf{F}^{-T}\nabla_X \hat{e} \cdot \textbf{F}^{-T}\nabla_X \hat{e}_p + J \hat{e}_p (\textbf{V}-\Tilde{\textbf{V}}) \cdot  \textbf{F}^{-T}\nabla_X \hat{e}\\
+ \int_{\partial K_{\mbox{\tiny in}}^t} J (\textbf{V}-\Tilde{\textbf{V}})\cdot \textbf{F}^{-T} [\![ \hat{e} ]\!] \hat{e}_h +  \varepsilon j_h(t;\hat{e},\hat{e}_p) + \Tilde{p}_h(\hat{e},\hat{e}_h)
 + \underset{\mathcal{E}_D}{\sum} \int_E \{\!\{\hat{\textbf{E}}_p\}\!\}\cdot J^\frac{1}{2} \textbf{F}^{-T}[\![ \hat{e}_h ]\!]\\
 + \frac{1}{2} \Vert \nabla\cdot \tilde{\textbf{V}} \Vert _{L^\infty(\Omega)}( \Vert e \Vert _{L^2(\Omega)}^2+ \Vert e_p \Vert _{L^2(\Omega)}^2 )
\end{multline}

\noindent Using $\hat{e}_h\! =\! \hat{e}\!-\!\hat{e}_p$, Young's inequality and bounding $I$ via $\alpha \! \geq \! 2C_T$. For positive $\alpha_i$:

\begin{multline}\label{ineq:finalapr}
    I \leq  [\frac{\alpha_1}{2}+\frac{\alpha_2}{2}+\frac{\alpha_3}{2}+\frac{\alpha_6}{4}(1+\theta)+\frac{\alpha_7}{4}(1+\theta)] \varepsilon |\hat{e}|^2_{U(t)} + [\frac{1}{2\alpha_2}+\frac{1}{4\alpha_8}] \varepsilon |\hat{e}_p|^2_{U(t)} \\
    + [ \frac{\alpha_4}{2}+  \frac{\alpha_5}{2}+\frac{(1+\theta)}{2\alpha_6}+\frac{\alpha_8}{2} +\frac{\alpha_9}{2}] \varepsilon j_h(t;\hat{e},\hat{e}) + [\frac{1}{2\alpha_5}+\frac{(1+\theta)}{2\alpha_7} +\frac{\alpha_{10}}{2}] \varepsilon j_h(t;\hat{e}_p,\hat{e}_p) \\
    + \frac{e^{4C_0^\Omega}}{4\alpha_4} \frac{1}{\varepsilon}\Vert \textbf{V}-\Tilde{\textbf{V}} \Vert _\infty^2 \Vert \hat{e}_h \Vert _{H(t)}^2 + \frac{\varepsilon}{4}[\alpha_9^{-1}+\alpha_{10}^{-1}] \hat{B}(\hat{\textbf{E}}_p)^2 + \frac{1}{\varepsilon} \Vert \textbf{V}-\Tilde{\textbf{V}} \Vert ^2_{L^\infty(\Omega)^2} \frac{1}{2\alpha_3}  \Vert e_p \Vert ^2_{L^2(\Omega)}\\
    + \frac{1}{\varepsilon} \Vert \textbf{V}-\Tilde{\textbf{V}} \Vert ^2_{L^\infty(\Omega)^2} \frac{1}{2\alpha_1}  \Vert e \Vert ^2_{L^2(\Omega)} +\frac{ 1}{2}\Vert \nabla\cdot \tilde{\textbf{V}} \Vert _{L^\infty(\Omega)}( \Vert e \Vert _{L^2(\Omega)}^2+ \Vert e_p \Vert _{L^2(\Omega)}^2 )
\end{multline}

\noindent is true for both SIPG and NIPG.

\noindent \textbf{For SIPG:}
$\alpha_6\!=\!\frac{3}{2}$, $\alpha_1\!=\!\alpha_2\!=\!\alpha_3\!=\!\alpha_7\!=\!\frac{1}{16}$, $\alpha_4\!=\!\alpha_5\!=\!\alpha_8\!=\!\alpha_9\!=\!\frac{1}{24}$, $\alpha_{10}\!=\!4$ implies
\begin{multline*}
\frac{1}{2} (\frac{d}{d t}  \Vert e \Vert _{L^2(\Omega)}^2 - \frac{d}{d t}  \Vert e_p \Vert _{L^2(\Omega)}^2) + \frac{\varepsilon}{8} (N(t,e)^2-120N(t,e_p)^2) - \frac{97}{16} \varepsilon \hat{B}(\hat{\textbf{E}}_p) \\
\leq ([8+12 e^{4C_0^\Omega} ]\frac{1}{\varepsilon} \Vert \textbf{V}-\Tilde{\textbf{V}} \Vert ^2_{L^\infty(\Omega)^2} +\frac{ 1}{2}\Vert \nabla\cdot \tilde{\textbf{V}} \Vert _{L^\infty(\Omega)})( \Vert e \Vert _{L^2(\Omega)}^2+ \Vert e_p \Vert _{L^2(\Omega)}^2 )
\end{multline*}

\noindent \textbf{For NIPG:}
$\alpha_1\!=\!\alpha_2\!=\!\alpha_3\!=\!\frac{1}{2}$, $\alpha_4\!=\!\alpha_5\!=\!\alpha_8\!=\!\alpha_9\!=\!\frac{1}{4}$, $\alpha_{10}\!=\!4$ implies
\begin{multline*}
\frac{1}{2} (\frac{d}{d t}  \Vert e \Vert _{L^2(\Omega)}^2 - \frac{d}{d t}  \Vert e_p \Vert _{L^2(\Omega)}^2) + \frac{\varepsilon}{4} (N(t,e)^2-8N(t,e_p)^2) - \frac{17}{16} \varepsilon \hat{B}(\hat{\textbf{E}}_p) \\
\hfill \leq ([1+2 e^{4C_0^\Omega} ]\frac{1}{\varepsilon} \Vert \textbf{V}-\Tilde{\textbf{V}} \Vert ^2_{L^\infty(\Omega)^2} +\frac{1}{2} \Vert \nabla\cdot \tilde{\textbf{V}} \Vert _{L^\infty(\Omega)})( \Vert e \Vert _{L^2(\Omega)}^2+ \Vert e_p \Vert _{L^2(\Omega)}^2 )\\
\mbox{and we conclude with Grönwall's inequality in the form: let }a,b,\alpha,\beta, C\geq0 \hfill \\
 \; \frac{d}{dt}(a-\alpha)+(b-\beta)\leq C(t)(a+\alpha) \mbox{ and } \mu(s,t)= \exp(\int_s^t C(y)dy) \mbox{ implies} \\
a(T) + \int_0^T\mu(s,T)b(s)ds \leq \mu(0,T)a(0) + 2\mu(0,T) \underset{0\leq s\leq T}{\max}\alpha(s) + \int_0^T \mu(s,T)\beta(s)ds
\end{multline*}
\end{proof}

\noindent Observe that in both cases, since $e^{4C_0^\Omega}$ remains close to 1 for short time intervals (where $t \nabla \cdot \tilde{\textbf{V}} = O(1)$), the multiplier ($C_S$ for SIPG and $C_N$ for NIPG) is driven exponentially by $ \Vert \textbf{V}-\Tilde{\textbf{V}} \Vert ^2/\varepsilon$ and $ \Vert \nabla \! \cdot \! \Tilde{\textbf{V}} \Vert $. This directs our study towards a balance between these terms. The following approximation properties yield the necessary tools to interpret the \textit{a priori} error estimate by providing upper bounds on the terms on the right hand side of \eqref{ineq:nipg} and \eqref{ineq:sipg}. In remark \ref{rq:apriori}, we will see that the actual balance is between $ \Vert \textbf{V}-\Tilde{\textbf{V}} \Vert_{L^\infty_2(\mathcal{Q}^T)^2} ^2/\varepsilon$ and $\Vert \Tilde{\textbf{V}} \Vert _{W^{3,\infty}_1(\mathcal{Q}^T)^2}$.

\begin{remark} \label{rq:velocity}
Suppose $\Tilde{\textbf{V}}\in W^{3,\infty}_1(\mathcal{Q}^T)^2$ and let $(\textbf{q}_K)_{K\in\mathcal{T}_h}$ be a collection of vectors such that $ \lVert \textbf{q}_K \rVert =1$. Furthermore, we write $\textbf{Y}_K=J^\frac{1}{2}\textbf{F}^{-T}\textbf{q}_K$. Then $\forall K, \; \textbf{Y}_K\in H^2(\mathcal{T}_h)^2$ and there is $\alpha_0, \alpha_1, \alpha_2>0$ such that:

\begin{align}
     \lVert \textbf{Y}_K \rVert _{L^2(K)^2}^2 &\leq|K| \exp(2\alpha_0 \Vert \Tilde{\textbf{V}} \Vert _{W^{1,\infty}_1(\mathcal{Q}^T)^2})\label{ineq:Y1}\\ 
    |\textbf{Y}_K|_{H^1(K)^2}^2 &\leq|K| \exp(2\alpha_1 \Vert \Tilde{\textbf{V}} \Vert _{W^{2,\infty}_1(\mathcal{Q}^T)^2})\label{ineq:Y2}\\ 
    |\textbf{Y}_K|_{H^2(K)^2}^2  &\leq|K| \exp(2C_2(t))\label{ineq:Y3}
\end{align}

\noindent with $C_2(t) = \alpha_2 \Vert \Tilde{\textbf{V}} \Vert _{W^{3,\infty}_1(\mathcal{Q}^T)^2}$. The proof is given in appendix, we just report that we used successive Grönwall's lemma.
\end{remark}

\begin{lemma}\label{lem:H2}
(Approximation in $H^2(\hat{\Omega})$) Let $T>0$, $1\leq p\leq 2$ and $\hat{u}_h$ be the solution of \eqref{weakform:discr}, $\hat{e}_p$ and $\hat{\textbf{E}}_p$ defined by \eqref{fct:Ep}, $N(t,\cdot)$ and $B(t,\cdot)$ be the seminorms defined by \eqref{norm:N} and \eqref{norm:B} respectively, $\Tilde{\textbf{V}}\in W^{3,\infty}_1(\mathcal{Q}^T)^2$ and $C_2(t)$ defined in \eqref{ineq:Y3}. Let ${h=\underset{E\in \mathcal{E}}{\max} \; h_E}$ then $\forall t \in [0,T]$

\begin{align}
N(t,e_p) &\lesssim e^{2C_0(t)+C_1(t)} h |\hat{u}(t,\cdot)|_{H^2(\hat{\Omega})}\label{ineq:a}\\
\label{ineq:b}\lVert \hat{e}_p \rVert _{L^2(\hat{\Omega})} &\lesssim e^{C_0(t)} h^2 |\hat{u}(t,\cdot)|_{H^2(\hat{\Omega})}\\
\label{ineq:c} B(t,\textbf{E}_p) &\lesssim h[e^{C_1(t)} |\hat{u}(t,\cdot)|_{H^2(\hat{\Omega})} + e^{C_2(t)}  \lVert \hat{u}(t,\cdot) \rVert _{L^2(\hat{\Omega})}] 
\end{align}

\end{lemma}

\begin{proof}
Let $t>0$, in this proof we take a snapshot of $\hat{u}$, namely when we write $\hat{u}$ we mean $\hat{u}(t,\cdot)\in H^2(\hat{\Omega})$.

\noindent \textbf{Proof of \eqref{ineq:a}:} See \cite{dolesji15} \textbf{Lemma 2.24}, resp. \textbf{2.25} state that: 

\begin{equation*}
|\hat{u}- P_h(0)\hat{u}|_{U(0)} \lesssim h |\hat{u}|_{H^2(\hat{\Omega})}, \mbox{ resp. } j_h(0;\hat{u}- P_h(0)\hat{u},\hat{u}- P_h(0)\hat{u}) \lesssim h^2 |\hat{u}|^2_{H^2(\hat{\Omega})}
\end{equation*}
    
\noindent thus along with \eqref{ineq:Unorm}-\eqref{ineq:jtime}: $\hat{N}\!(t;\hat{u}- P_h\!(0)\hat{u}) \! \lesssim \! e^{C_1\!(t)} \! \hat{N}\!(0;\hat{u}- P_h\!(0)\hat{u}) \! \lesssim \! e^{C_1\!(t)} \! h  \vert \hat{u}\vert _{H^2(\hat{\Omega})}$.

\begin{align*}
\mbox{Also by \eqref{ppt:ie}:}|P_h(0)\hat{u}- P_h(t)\hat{u}|_{U(0)} & \!\leq\! \frac{1}{h}C_I  \lVert P_h(t)\hat{u}- P_h(0)\hat{u} \rVert _{H(0)} \\
& \!\leq\! \frac{1}{h}C_Ie^{C_0(t)} ( \lVert P_h(t)\hat{u} \!-\! \hat{u} \rVert _{H(t)} \! +\!   \lVert P_h(0)\hat{u}\! -\! \hat{u} \rVert _{H(t)}) \\
& \!\leq\! 2\frac{1}{h}C_Ie^{2C_0(t)}  \lVert P_h(0)\hat{u}- \hat{u} \rVert _{H(0)} 
\end{align*}

\begin{equation*}
\mbox{and by \eqref{ppt:ite}: }j_h(0;P_h(0)\hat{u}- P_h(t)\hat{u},P_h(0)\hat{u}- P_h(t)\hat{u})  \leq 2\frac{C_T}{h^2}  \lVert P_h(t)\hat{u}- P_h(0)\hat{u} \rVert ^2_{H(0)}.
\end{equation*}

\begin{align*}
\mbox{Therefore: } N(t,e_p)  & \leq \hat{N}(t,\hat{u}- P_h(0)\hat{u}) + \hat{N}(t,P_h(t)\hat{u}- P_h(0)\hat{u}) \\ 
& \lesssim e^{C_1(t)} h |\hat{u}|_{H^2(\hat{\Omega})} + e^{C_1(t)}\hat{N}(0,P_h(t)\hat{u}- P_h(0)\hat{u}) \\
& \lesssim e^{C_1(t)} h |\hat{u}|_{H^2(\hat{\Omega})} + \frac{1}{h}e^{2C_0(t)+C_1(t)} \lVert P_h(t)\hat{u}- P_h(0)\hat{u} \rVert _{H(0)} \\
& \lesssim e^{2C_0(t)+C_1(t)} h |\hat{u}|_{H^2(\hat{\Omega})}
\end{align*}
 
\begin{equation*}
\mbox{\textbf{Proof of \eqref{ineq:b}:} }\lVert \hat{e}_p \rVert\! _{H\!(t)} \! \leq \! \lVert \hat{u}-P_h(0)\hat{u} \rVert\! _{H\!(t)} \! \leq \! e^{C_0\!(t)} \! \lVert \hat{u}-P_h\! (0)\hat{u} \rVert \! _{H\!(0)}\! \lesssim\! e^{C_0\!(t)} \! h^2 |\hat{u}| \! _{H^2(\hat{\Omega})}.
\end{equation*}

\noindent \textbf{Proof of \eqref{ineq:c}:} Let $\Pi_{1}$ the $L^2$-projection on the DG space with polynomial degree 1. Then

\begin{multline*}
\hat{\textbf{E}}_p = \underset{\textbf{A}_1}{\underbrace{J^\frac{1}{2}\textbf{F}^{-T} \nabla_X (\hat{u}-\Pi_1\hat{u})}} + \underset{\textbf{A}_2}{\underbrace{J^\frac{1}{2}\textbf{F}^{-T}\nabla_X\Pi_1\hat{u}-\Pi_{L^2}(J^\frac{1}{2}\textbf{F}^{-T}\nabla_X\Pi_1\hat{u})}} \\ 
+ \underset{\textbf{A}_3}{\underbrace{\Pi_{L^2}(J^\frac{1}{2}\textbf{F}^{-T}\nabla_X(\Pi_1\hat{u}-\hat{u}))}}
\end{multline*}

\noindent By \textbf{Lemma 2.25} in \cite{dolesji15}: $\hat{B}(\textbf{A}_1)\leq e^{C_1(t)} \hat{B}(\nabla_X (\hat{u}-\Pi_1\hat{u})) \lesssim  e^{C_1(t)} h|\hat{u}|_{H^2(\hat{\Omega})}$.

\noindent Additionally $\nabla_X\Pi_1\hat{u}$ is constant over every triangle, we write $\nabla_X\Pi_1\hat{u} = \rho_K \textbf{q}_K$ with $\rho_K>0$ and $ \lVert \textbf{q}_K \rVert =1$, implying
    
\begin{equation*}
\rho_K^2 h_K^2 \approx \lVert \nabla_X\Pi_1\hat{u} \rVert ^2_{L^2(K)} \lesssim \frac{1}{h_K^2} \lVert \Pi_1\hat{u} \rVert ^2_{L^2(K)^2} \lesssim \frac{1}{h_K^2}  \lVert \hat{u} \rVert ^2_{L^2(K)}
\end{equation*}
    
\noindent By \textbf{Lemma 2.25} in \cite{dolesji15} and \eqref{ineq:Y3}:

\begin{equation*}
\hat{B}(\textbf{A}_2)^2  \lesssim  \underset{K}{\sum} \lVert \hat{u} \rVert ^2_{L^2(K)} |J^\frac{1}{2}\textbf{F}^{-T}\textbf{q}_K|^2_{H^2(K)} \lesssim  e^{2C_2(t)} h^2  \lVert \hat{u} \rVert _{L^2(\hat{\Omega})}^2
\end{equation*}

\noindent By \eqref{ppt:ite}, since $A_3$ is piecewise polynomial and by \eqref{ppt:L2}

\begin{equation*}
\hat{B}(\textbf{A}_3)^2 
 \lesssim  \underset{K}{\sum}\int_{K} \textbf{A}_3^2 \lesssim \underset{K}{\sum}\int_{K} \textbf{A}_1^2  \lesssim e^{2C_1(t)} |\Pi_1\hat{u}-\hat{u}|^2_{U(0)} \lesssim e^{2C_1(t)} h^2 |\hat{u}|^2_{H^2(\hat{\Omega})}
\end{equation*}

\noindent Adding the three contributions $\hat{B}(\textbf{A}_1)$, $\hat{B}(\textbf{A}_2)$, $\hat{B}(\textbf{A}_3)$ yields to \eqref{ineq:c}.

\end{proof}

\begin{remark}\label{rq:apriori}
    The combination of theorem \ref{thm:apriori} and lemma \ref{lem:H2} gives for $1\leq p\leq 2$: 

\begin{equation*}
     \lVert e \rVert _{L^\infty(0,T;L^2(\Omega))}\leq \Tilde{C} h^{p-1}(\sqrt{\varepsilon}+h),
\end{equation*}

\noindent $\Tilde{C}$ depending on $|\hat{u}|_{H^2_1(\hat{\mathcal{Q}}^T)}$, $\lVert \! \Tilde{\textbf{V}} \rVert _{W^{3,\infty}_1(\mathcal{Q}^T)^2}$ and $\frac{1}{\varepsilon} \lVert \! \textbf{V}-\Tilde{\textbf{V}} \rVert ^2_{L^\infty_2(\mathcal{Q}^T)^2}$.
The error is decomposed into three distinct components: the local error in $\hat{u}$, the error component influenced by the mesh movement $\Tilde{\textbf{V}}$ and the error from the deviation from the mean flow $\textbf{V}-\Tilde{\textbf{V}}$. It proves that the mesh velocity $\tilde{\textbf{V}}$ has to be a smoothed approximation of the flow velocity $\textbf{V}$ such that both $\lVert \tilde{\textbf{V}}\rVert_{W^{3,\infty}(\Omega)^2}$ and $\frac{1}{\varepsilon} {\lVert \textbf{V}-\tilde{\textbf{V}}\rVert^2_{W^{0,\infty}(\Omega)^2}} $ are bounded.

This estimate also underlines a second order convergence when the remaining advection speed is 0. This is a higher order than the one proven in \cite{dolesji15}. However when the remaining advection speed is not reduced enough, the half less order of convergence from \cite{dolesji15} stands. 

The dependence in $\lVert \! \Tilde{\textbf{V}} \rVert _{W^{3,\infty}_1(\mathcal{Q}^T)^2}$ also implies that even if the scheme does not theoretically needs to keep $T$ small for stability or convergence (with, for instance, a remeshing step), it is still optimal to do so for precision needs.
\end{remark}

\subsection[An \textit{a posteriori} error estimate]{An \textit{a posteriori} error estimate for the semi-discrete formulation}

\noindent This subsection deals with an \textit{a posteriori} approach to the semi-discretisation. Theorem \ref{thm:apstsp-ti} concludes to the reliability of the local spatial criteria defined in \eqref{stestim}. This subsection is then followed by the proof of theorem \ref{thm:apststd} that states the reliability in space of the weighted criterion \eqref{west}.

This \textit{a posteriori} error criteria hold only when $h_E \ll \varepsilon$ (see \cite{schotzau09}).

\noindent In this section $\hat{u}$ is the solution of \eqref{weakform:ctn} and $\hat{u}_h$ the solution of \eqref{weakform:discr}. Finally we define $\hat{u}^s \colon [0;T] \to H_D^1(\hat{\Omega})$ the pointwise solution of the space-discrete problem:

\begin{equation}\label{eq:us}
\forall \hat{v} \in H_D^1(\hat{\Omega}), \;  a(t;\hat{u}^s(t),\hat{v})=l(\hat{v})-\int_{\hat{\Omega}} J\frac{\partial \hat{u}_h}{\partial t}\hat{v}
\end{equation}

\noindent and suppose $\hat{u}^s(t,\cdot) \in H^2(\hat{\Omega})$. Since $\chi(t,\cdot) \in C^2(\hat{\Omega},\Omega)$ then $\textbf{F}\in C^1(\hat{\Omega},\mathbb{R}^{2\times 2})$ and the regularity theorems from \cite{quarteroni94} can be used.

We define $\hat{e}=\hat{u}-\hat{u}_h=\hat{\rho}+\hat{\sigma}$ with $\hat{\rho}= \hat{u}-\hat{u}^s$ and $\hat{\sigma} = \hat{u}^s-\hat{u}_h$. The final estimate \eqref{estim} is based on theorem \ref{thm:apststd} and lemma \ref{lem:jump} and is proven in section 4.3. Following the definitions \eqref{elts}, let $\rho_S \! =\! \underset{S}{min}(\frac{a_\infty^{\omega_S}h_S}{\sqrt{\varepsilon}},\!\frac{M_\infty^{\omega_S}}{\beta})$ and

\begin{align*}\eta^{t\;2}_{J_K} & = \frac{1}{2} \underset{E\in \mathcal{E}^{\mbox{\tiny int}}\cap K}{\sum} \beta h_E J_1^{\omega_E} \int_E [\![ \hat{u}_h ]\!]^2 + (\frac{\delta_\infty^{\omega_E}}{\varepsilon} h_E J_\infty^{\omega_E}+a_1^{\omega_E} \frac{\varepsilon\alpha }{h_E} )a_\infty^{\omega_E} \int_E J(\textbf{F}^{-T} [\![ \hat{u}_h ]\!] )^2 \\
&  + \underset{E\in \mathcal{E}_D^{\mbox{\tiny ext}}\cap K}{\sum} \beta h_E J_1^{\omega_E} \int_E [\![ \hat{u}_h ]\!]^2 + (\frac{\delta_\infty^{\omega_E}}{\varepsilon} h_E J_\infty^{\omega_E}+a_1^{\omega_E} \frac{\varepsilon\alpha }{h_E} ) a_\infty^{\omega_E} \int_E J(\textbf{F}^{-T} [\![ \hat{u}_h ]\!] )^2 \\
\eta^{t\;2}_{E_K} & \!= \frac{1}{2} \underset{E \in \mathcal{E}^{\mbox{\tiny int}}\cap K}{\sum} \hspace{-3mm} \rho_E \sqrt{\frac{a_\infty^{\omega_E}}{\varepsilon}} \int_E (\varepsilon J\textbf{F}^{-1}\textbf{F}^{-T} [\![ \nabla\! _X \hat{u}_h ]\!] )^2 \\
& + \underset{E\in \mathcal{E}_N^{\mbox{\tiny ext}}\cap K}{\sum}\hspace{-3mm} \rho_E \sqrt{\frac{a_\infty^{\omega_E}}{\varepsilon}}  \int_E (\hat{u}_N \! - \! \varepsilon J\textbf{F}^{-1}\textbf{F}^{-T} \nabla\! _X \hat{u}_h \cdot \textbf{n})^2 \\
\eta^{t\;2}_{R_K} & = \rho_K^2  \lVert J\hat{f} \! - \! J\frac{\partial \hat{u}_h}{\partial t}  \! + \! \varepsilon  \nabla\! _X \!\cdot\!  \{J \textbf{F}^{-1} \textbf{F}^{-T} \nabla _X \hat{u}_h\}\!-\! J (\textbf{V} \! - \! \tilde{\textbf{V}})\! \cdot \! \textbf{F}^{-T}\nabla _X \hat{u}_h\rVert^2_{L^2(K)} 
\end{align*}

\begin{equation}\label{west}
    \eta^{t\;2}_K  = (\eta^{t\;2}_{J_K} + \eta^{t\;2}_{R_K} + \eta^{t\;2}_{E_K})
\end{equation}

\noindent With these criteria we obtain a way to measure the deviation from the mean flow $\Tilde{\textbf{V}}$: the error is bounded by the criteria composed of the residuals weighted by the physical properties of the flow map ($a$, $J$ and $M$). In particular $\eta_J$ and $\eta_E$ measure the deviation to a solution $\hat{u}^s \in H^2(\hat{\Omega})$ and $\eta_R$ measures the deviation to the mean flow. Taking the physical properties of the flow map, allows us to measure the effect of potential mesh entanglement and modification of the mesh quality.

\noindent The following theorem states the reliability of the criteria.

\begin{theorem}\label{thm:apststd} Let $T>0$, $\hat{u}_h$ be the solution of \eqref{weakform:discr} and $\hat{u}^s\colon [0;T] \to H_D^1(\hat{\Omega})$ be the pointwise solution of \eqref{eq:us}, such that $\forall t\in [0,T]\; \hat{u}^s(t) \in H^2(\hat{\Omega})$, let $\hat{\sigma} = \hat{u}^s-\hat{u}_h$ and $\eta^{t}_K$ be defined in \eqref{west} and $\alpha>0$ the stability constant in the definition of $j_h$. Then $\forall t\in [0,T]$
\begin{equation}
    (|||\hat{\sigma}|||+|\hat{\sigma}|_A)^2 \lesssim \underset{K \in \mathcal{T}_h}{\sum}  (1+\frac{1}{\alpha}) \eta^{t\;2}_K
\end{equation}
\end{theorem}

\begin{proof}
    The proof is the subject of subsection \ref{ssec:proof} (see lemma \ref{lemthm1}).
\end{proof}

\noindent In the following lemma, we use the projection operator $\mathcal{A}_h$ defined in lemma \ref{lem3}.

\begin{lemma}\label{lem:jump}
Let $\hat{u}_h^r=\hat{u}_h -\mathcal{A}_h \hat{u}_h$. For all $t \in [0,T]$ 
\begin{align*}
      |||\hat{u}_h^r|||^2 + |\hat{u}_h^r|_A^2 \lesssim & \underset{K \in \mathcal{T}_h}{\sum} [\frac{1}{\alpha} + 1] \eta^{t\;2}_{J_K}  \\
      \lVert \hat{u}_h^r \rVert _{H(t)}^2 \lesssim & \underset{E \in \mathcal{E}_D}{\sum} h_E J_1^{\omega_E} \int_E  [\![ \hat{u}_h ]\!] ^2\\
      \lVert \frac{\partial \hat{u}_h^r}{\partial t} \rVert _{H(t)}^2 \lesssim & \underset{E \in \mathcal{E}_D}{\sum} h_E J_1^{\omega_E} \int_E  (\frac{\partial[\![  \hat{u}_h ]\!]}{\partial t}) ^2
\end{align*}
\end{lemma}

\begin{proof}
    The first two inequalities are the subject of lemma \ref{lem4} ; the third one is concluded from remark \ref{rqAh} that implies: $\forall \hat{v}_h \in C^1(0,T;V_h), \hspace{3mm}\frac{\partial }{\partial t}(\mathcal{A}_h \hat{v}_h)=\mathcal{A}_h\frac{\partial \hat{v}_h}{\partial t}$, and then applying the lemma \ref{lem3} to the function $\frac{\partial \hat{u}_h}{\partial t}$.
\end{proof}

\begin{lemma}\label{lem7}
$\forall \hat{v} \in H_D^1(\hat{\Omega}),\; \int _{\hat{\Omega}}J\frac{\partial \hat{e}}{\partial t} \hat{v} + a(t;\hat{\rho},\hat{v})=0$
\end{lemma}

\begin{proof}
The lemma is a direct consequence of the definition of $\hat{u}$ and $\hat{u}^s$.
\end{proof}

$\hat{u}^s$ is a time dependent function but we will see in subsection \ref{ssec:proof} that by fixing $t\in [0,T]$ we always consider the steady-state problem \eqref{eq:us} to prove theorem \ref{thm:apststd}. Since theorem \ref{thm:apststd} splits space and time, we use this orthogonality lemma \ref{lem7} to assemble space and time. This lemma will enable us to include the properties from the steady-state problem \eqref{eq:us} to theorem \ref{thm:apstsp-ti}, that is true for the solutions $\hat{u}$ and $\hat{u}_h$ of the unsteady problems \eqref{weakform:ctn} and \eqref{weakform:discr}.

To state the final global error estimate we need some criteria. Let

\begin{align}
    \eta^{t\;2}_1 & =  \underset{K \in \mathcal{T}_h}{\sum}  (1+\frac{1}{\alpha}) \eta^{t\;2}_K \nonumber\\
\eta^{t\;2}_2 & =  \underset{E \in \mathcal{E}_D}{\sum} h_E J_1^{\omega_E} \int_E  (\frac{\partial[\![  \hat{u}_h ]\!]}{\partial t}) ^2 \label{stestim}\\
\eta^{t\;2}_3 & = \underset{E \in \mathcal{E}_D}{\sum} h_E J_1^{\omega_E} \int_E  [\![ \hat{u}_h ]\!] ^2 \nonumber
\end{align}

For $\hat{v} \in H^1_\infty([0,T]\times \mathcal{T}_h)$ and $v(t,x) = \hat{v}(t,X)$ we define:

\begin{equation}\label{stnorm}
     \lVert v \rVert _\#^2 =  \lVert v \rVert _{L^2_\infty(\mathcal{Q}^T)}^2 + \int_0^T |||\hat{v}|||^2 d  t
\end{equation}

\noindent This allows us to prove the final estimate.
\begin{theorem}\label{thm:apstsp-ti}
Let $\hat{u} \in H^1_{D\,2}(\hat{\mathcal{Q}}^T) \cap L^\infty_\infty(\hat{\mathcal{Q}}^T)$ be a solution of \eqref{weakform:ctn}, $\hat{u}_h \in C^1(0,T;V_h)$ a solution of \eqref{weakform:discr}, $\eta^t _1$, $\eta^t _2$ and $\eta^t _3$ be defined in \eqref{stestim}, $ \lVert \cdot \rVert _\#^2$ be the norm defined in \eqref{stnorm}, $e = (\hat{u}-\hat{u}_h)(t, \chi (t,X))$. Then 

\begin{equation}\label{estim}
    \lVert e \rVert _\#^2 \lesssim S_0(t) \{ \lVert e(0) \rVert _{L^2(\Omega)}^2 + \int_0^T \eta^{t\;2}_1 + T \int_0^T \eta^{t\;2}_2 + \underset{t\in[0,T]}{\max}(\eta^{t\;2}_3) \}
\end{equation}

\noindent holds with $S_0(t)= \exp (2C_0(t))$.

\end{theorem}

\begin{proof}
Let $\hat{e}_c= \hat{u}-\hat{u}_h^c \in H_D^1(\hat{\Omega})$ for all $t$. Testing lemma \ref{lem7} with $\hat{e}_c$ we have

\begin{equation*}
    \int _{\hat{\Omega}}J\frac{\partial \hat{e}_c}{\partial t} \hat{e}_c + a(t;\hat{e}_c,\hat{e}_c) = \int _{\hat{\Omega}}J\frac{\partial \hat{u}_h^r}{\partial t} \hat{e}_c + a(t;\hat{\sigma}_c,\hat{e}_c)
\end{equation*}

\noindent Additionally the inequalities 

\begin{align*}
     & \int _{\hat{\Omega}}J\frac{\partial \hat{e}_c}{\partial t} \hat{e}_c \geq \frac{1}{2} \frac{d }{d  t}  \lVert e_c \rVert _{L^2(\Omega)}^2 - \frac{ 1}{2}\lVert \nabla\cdot \tilde{\textbf{V}} \rVert _\infty \lVert e_c \rVert _{L^2(\Omega)}^2 \mbox{, } \\
     & a(t;\hat{e}_c,\hat{e}_c) \geq |||\hat{e}_c|||^2, \; \int _{\hat{\Omega}}J\frac{\partial \hat{u}_h^r}{\partial t} \hat{e}_c \leq \frac{T}{2}  \lVert \frac{\partial \hat{u}_h^r}{\partial t} \rVert _{H(t)}^2 + \frac{1}{2T}  \lVert e_c \rVert _{L^2(\Omega)}^2 \mbox{ and }\\
     & \hfill a(t;\hat{\sigma}_c,\hat{e}_c) \leq C \cdot (|||\hat{\sigma}_c|||+|\hat{\sigma}_c|_A)|||\hat{e}_c|||  \leq  \frac{C^2}{2}  \cdot (|||\hat{\sigma}_c|||+|\hat{\sigma}_c|_A)^2 + \frac{|||\hat{e}_c|||^2}{2} 
\end{align*}

\noindent hold. Thus

\begin{equation*}
    \frac{d }{d  t}  \lVert e_c \rVert _{L ^2(\Omega)}^2 + (|||\hat{e}_c|||^2  - C^2 \cdot (|||\hat{\sigma}_c|||  + |\hat{\sigma}_c|_A )^2 - T \cdot  \lVert \frac{\partial \hat{u}_h^r}{\partial t} \rVert _{H (t)}^2)  \!\leq\!  ( \lVert \nabla \cdot \tilde{\textbf{V}} \rVert _\infty + \frac{1}{T})  \lVert e_c \rVert _{L ^2 (\Omega)}^2.
\end{equation*}

\noindent By Gronwall's lemma 
\begin{equation*}
\lVert e_c \rVert _\#^2  \lesssim S_0 (t) \{  \lVert e_c(0) \rVert _{L^2(\Omega)}^2  + \int_0^T  (|||\hat{\sigma}_c|||  + |\hat{\sigma}_c|_A)^2  + T  \int_0^T   \lVert \frac{\partial \hat{u}_h^r}{\partial t} \rVert _{H (t)}^2 \}.
\end{equation*}

\noindent And the definition of $\eta^t_3$ gives \eqref{estim}.
\end{proof}

\noindent This theorem states the reliability of the space-time criteria $\eta^t_1$, $\eta^t_2$ and $\eta^t_3$ and it uses theorem \ref{thm:apststd} that states the reliability of the steady-state criterion $\eta_K$ and lemma \ref{lem7} that states the orthogonality of $\hat{\rho}$.

\noindent Theorem \ref{thm:apststd} is proven in the following section.

\subsection{Proof of Theorem \ref{thm:apststd}}\label{ssec:proof}

The topic of this section is to prove the reliability of the error estimate exposed in theorem \ref{thm:apststd}. Based on the continuity and inf-sup conditions given in Section \ref{ssec:prop}, we show that a steady-state form of the weighted estimator \eqref{west} solves theorem \ref{thm:apststd}. The first lemmas of this section give an upper bound for each contribution of the operator $\Tilde{a}_h$, the last argument uses the inf-sup condition and an upper bound for the operator $\Tilde{p}_h$.

\noindent The outline of the proof for the stationary case is as follows: separate our solution into a continuous and a discontinuous part (lemma \ref{lem3}), give a bound to the discontinuous part (lemma \ref{lem4}), derive a bound for the bilinear forms as an estimate multiplied by the energy-norm of the continuous function (lemma \ref{lem2}, lemma \ref{lem6}) and conclude using lemma \ref{lem:infsup}.

In this subsection, we just fix some $t\in [0,T]$ in order to prove theorem \ref{thm:apststd}.

First we give the upper bound on the operator $\Tilde{p}_h$ for test functions in $V_h^c$.

\begin{lemma}\label{lem2}$\forall \hat{v}\in V_h \hspace{1mm} \forall \hat{w} \in V_h^c $

\begin{equation*}
    |\Tilde{p}_h(\hat{v},\hat{w})| \lesssim \alpha^{-\frac{1}{2}}(\underset{E\in \mathcal{E}_D}{\sum} a_1^{\omega_E} a_\infty^{\omega_E} \int_E \frac{\varepsilon \alpha}{h_E} J(\textbf{F}^{-T} [\![ \hat{v} ]\!] )^2)^{\frac{1}{2}} \cdot (\underset{K\in \mathcal{T}_h}{\sum} \frac{\varepsilon}{a_1^{K} a_\infty^{\omega_K}}  |\hat{w}| ^2_{U_K(t)})^\frac{1}{2}
\end{equation*}
\end{lemma}

\begin{proof}
We have $\Tilde{p}_h(\hat{v},\hat{w}) = - \underset{E \in \mathcal{E}_D}{\sum} \varepsilon \int_E  \{\!\{\Pi_{L^2} ( J^\frac{1}{2}\! \textbf{F}^{-T} \nabla\! _X \hat{w}) \}\!\} \cdot J^\frac{1}{2}\! \textbf{F}^{-T} [\![  \hat{v}  ]\!] $

\noindent And by the Cauchy-Schwarz inequality

\begin{align*}
    |\Tilde{p}_h(\hat{v},\hat{w})| & \! \lesssim \! (\! \underset{E \in \mathcal{E}_D}{\sum} \! a_1^{\omega_E}\! a_\infty^{\omega_E} \! \int_E \! \frac{\varepsilon \alpha}{h\! _E} J\! (\textbf{F}^{-T} \! [\![ \hat{v} ]\!] )^2)\! ^\frac{1}{2} \! (\! \underset{K\in \mathcal{T}_h}{\sum} \! \int_{\partial K} \! \frac{\varepsilon h_K}{2 \alpha a_1^{K}\! a_\infty^{\omega_K}} \Pi_{L^2} \! ( J^\frac{1}{2}\! \textbf{F}^{-T}\! \nabla\! _X \hat{w})|_K^2\! )\! ^\frac{1}{2} \\
    &  \lesssim \alpha^{-\frac{1}{2}} (\underset{E \in \mathcal{E}_D}{\sum} a_1^{\omega_E} a_\infty^{\omega_E} \int_E \frac{\varepsilon \alpha}{h_E} J(\textbf{F}^{-T} [\![ \hat{v} ]\!] )^2)^\frac{1}{2}(\underset{K\in \mathcal{T}_h}{\sum} \frac{\varepsilon}{a_1^{K} a_\infty^{\omega_K}} |\hat{w}| ^2_{U_K(t)})^\frac{1}{2} 
\end{align*}

\end{proof}

\noindent We will now study an approximation of elements of $V_h$ by elements of $V_h^c$ in case of conforming meshes. A similar theorem for the Eulerian problem is stated in  \textbf{theorem 2.2} in \cite{karakashian03}.

\begin{lemma}\label{lem3}
Let $\mathcal{T}_h$ be a conforming mesh. Then there exists an approximation operator $\mathcal{A}_h \colon V_h \rightarrow V_h^c$ satisfying:

\begin{equation*}
    \forall \hat{v}_h \in V_h 
    \left \{ \begin{array}{ll}
         \underset{K\in \mathcal{T}_h}{\sum}  \lVert \hat{v}_h- \mathcal{A}_h \hat{v}_h \rVert ^2_{H_K} & \lesssim \underset{E\in \mathcal{E}_D}{\sum}   h_E J_1^{\omega_E} \int_E [\![  \hat{v}_h ]\!] ^2\\
         \underset{K\in \mathcal{T}_h}{\sum} |\hat{v}_h- \mathcal{A}_h \hat{v}_h|^2_{U_K} & \lesssim \underset{E\in \mathcal{E}_D}{\sum}   h_E^{-1} a_1^{\omega_E} \int_E [\![  \hat{v}_h ]\!] ^2
    \end{array}
    \right .
\end{equation*}
\end{lemma}

\begin{proof}

We build the approximation with Lagrangian nodes and use the property of polynomials to conclude.

\noindent For each $K\in \mathcal{T}_h$, let $\mathcal{N}_K = \{x_K^{(j)}\hspace{0.5mm}\colon\hspace{0.5mm} j=1,\dots , m\}$ be the set of distinct nodes of $K$ with nodes on both ends. Let $\{\phi_K^{(j)}\hspace{0.5mm}\colon\hspace{0.5mm} j=1,\dots , m\}$ be a local basis of functions satisfying $\phi_K^{(i)}(x_K^{(j)})=\delta_{ij}$, the Kronecker-$\delta$. Let $\mathcal{N} = \underset{K\in \mathcal{T}_h}{\bigcup}\mathcal{N}_K$ be the set of nodes and

\begin{align*}
    &\mathcal{N}_D = \{\nu \in \mathcal{N} \hspace{1mm}\colon\hspace{1mm} \nu \in E \in \mathcal{E}^{\mbox{\tiny ext}}_D\}, & \mathcal{N}_N = \{\nu \in \mathcal{N} \hspace{1mm}\colon\hspace{1mm} \nu \in E \in \mathcal{E}^{\mbox{\tiny ext}}_N\} \\
    &\mathcal{N}_i = \{ \nu \in \mathcal{N}\setminus\mathcal{N}_D \hspace{1mm}\colon\hspace{1mm} |\omega_\nu|=1\}, &\mathcal{N}_v = \mathcal{N}\setminus(\mathcal{N}_i \cup \mathcal{N}_D)
\end{align*}

\noindent For each  $\nu \in \mathcal{N}$, we define $\omega_\nu = \{K\in \mathcal{T}_h \hspace{0.5mm}\colon\hspace{0.5mm} \nu \in K \}$ . $|\omega_\nu|$ is uniformly bounded by a constant depending only on $\xi_0$. Finally, let $\bar{\mathcal{N}}$ be the collection of distinct Lagrange nodes $\nu$ needed to build an element of $V_h^c$. In this case (conforming mesh) we have $\bar{\mathcal{N}}=\mathcal{N}$. To each $\nu \in \bar{\mathcal{N}}$ we associate a basis function $\phi^{(\nu)}$:

\begin{equation*}
    \mbox{supp} \hspace{0.5mm} \phi^{(\nu)} \subset \underset{K \in \omega_\nu}{\bigcup} K, \; \phi^{(\nu)}|_{K} = \phi_K^{(j)}, \hspace{1mm} x_K^{(j)}=\nu.
\end{equation*}

\noindent Write $\hat{v}_h \in V_h$ as $\hat{v}_h = \underset{K \in \mathcal{T}_h}{\sum} \underset{j=1}{\overset{m}{\sum}} \alpha^{(j)}_K\phi^{(j)}_K$ and define

\begin{equation*}
    \mathcal{A}_h \hat{v}_h = \underset{\nu \in \bar{\mathcal{N}}}{\sum} \beta ^{(\nu)}\phi ^{(\nu)}\mbox{, where } \beta ^{(\nu)} = \left \{ \begin{array}{ll}
         0 & \mbox{ if } \nu \in \mathcal{N}_D \\
         \frac{1}{|\omega_\nu|} \sum_{x^{(j)}_K=\nu} \alpha_K^{(j)} & \mbox{ if } \nu \in \bar{\mathcal{N}}\setminus\mathcal{N}_D
    \end{array} \right.     
\end{equation*}

\noindent $\mathcal{A}_h \hat{v}_h$ is a continuous polynomial that interpolates $\hat{v}_h$ in the inner Lagrangian nodes of the mesh. For the Lagrangian nodes lying on the vertices and edges of the mesh it has the value of the average values of the neighbouring cells.

\noindent We define now $\beta_K^{(j)} \! =\! \beta^{(\nu)}$ if $x_K^{(j)}\! =\! \nu$. We have $\lVert \phi_K^{(j)} \rVert ^2_{L\! ^\infty\! (K)} \! \lesssim \! 1$ and $\lVert \nabla \phi_K^{(j)} \rVert ^2_{L\! ^\infty\! (K)} \! \lesssim \! h_K^{-2}$.

\begin{align*}
    \mbox{Thus: }\underset{K \in \mathcal{T}_h}{\sum} |\hat{v}_h - \mathcal{A}_h \hat{v}_h|^2_{U_K} & \lesssim \underset{K \in \mathcal{T}_h}{\sum} a_1^{K} \; h_K^2 \; |\nabla_X(\hat{v}_h - \mathcal{A}_h \hat{v}_h)|^2_{L^\infty(K)} \\
    & \lesssim \underset{K \in \mathcal{T}_h}{\sum}  a_1^{K} \underset{j=1}{\overset{m}{\sum}} |\alpha_K^{(j)} - \beta_K^{(j)}|^2  \lesssim \underset{\nu \in \mathcal{N}}{\sum}  a_1^{\omega_\nu}  \underset{x^{(j)}_K=\nu}{\sum} |\alpha_K^{(j)} - \beta^{(\nu)}|^2 \\
    & \lesssim \underset{\nu \in \mathcal{N}_v}{\sum} a_1^{\omega_\nu} \underset{x^{(j)}_K=\nu}{\sum} |\alpha_K^{(j)} - \beta^{(\nu)}|^2 + \underset{\nu \in \mathcal{N}_D}{\sum} a_1^{\omega_\nu} \underset{x^{(j)}_K=\nu}{\sum} |\alpha_K^{(j)}|^2 \\
    \underset{K \in \mathcal{T}_h}{\sum} \lVert \hat{v}_h - \mathcal{A}_h \hat{v}_h\rVert ^2_{H_K} & \lesssim \underset{\nu \in \mathcal{N}_v}{\sum} J_1^{\omega_\nu} h_\nu^2 \underset{x^{(j)}_K=\nu}{\sum} |\alpha_K^{(j)} - \beta^{(\nu)}|^2 + \underset{\nu \in \mathcal{N}_D}{\sum} J_1^{\omega_\nu} h_\nu^2 \underset{x^{(j)}_K=\nu}{\sum} |\alpha_K^{(j)}|^2
\end{align*}

\noindent where we used $\alpha_K^{(j)}=\beta^{(\nu)}$ for $\nu \in \mathcal{N}_i$. 

\noindent Now let $A = \underset{\nu \in \mathcal{N}_v}{\sum} a_1^{\omega_\nu} \underset{x^{(j)}_K=\nu}{\sum} |\alpha_K^{(j)} - \beta^{(\nu)}|^2  $ and $B = \underset{\nu \in \mathcal{N}_D}{\sum} a_1^{\omega_\nu} \underset{x^{(j)}_K=\nu}{\sum} |\alpha_K^{(j)}|^2$.

\noindent Let $\nu \in \mathcal{N}_v$ s.t. $\omega_\nu =\{ K_1,\dots, K_{|\omega_\nu|}\}$ with $\mu_{d-1}(K_l \cap K_{l+1})>0$ then there is a constant depending only on $|\omega_\nu|$ and so $\xi_0$ such that

\begin{equation*}
    \underset{x^{(j)}_K=\nu}{\sum} |\alpha_K^{(j)} - \beta^{(\nu)}|^2 \leq c \underset{l=1}{\overset{|\omega_\nu|-1}{\sum}} |\alpha_{K_l}^{(j_l)} - \alpha_{K_{l+1}}^{(j_{l+1})}|^2.\mbox{ Then } A \lesssim \underset{\nu \in \mathcal{N}_v}{\sum}  a_1^{\omega_\nu} \underset{l=1}{\overset{|\omega_\nu|-1}{\sum}} |\alpha_{K_l}^{(j_l)} - \alpha_{K_{l+1}}^{(j_{l+1})}|^2.
\end{equation*}

\noindent Similarly when $\nu \!\in\! \mathcal{N}_D$ and $|\omega_\nu|\! >\! 1$ s.t. $\omega_\nu \! = \! \{ K_1,\dots, K_{|\omega_\nu|}\}$ with $\mu_{d-1}(K_l \cap K_{l+1})\! >\! 0$:

\begin{align*}
    & \underset{x_K^{(j)}=\nu}{\sum}\! |\alpha_K^{(j)}|^2\! \lesssim\! \underset{l=1}{\overset{|\omega_\nu|\! -\! 1}{\sum}}\! |\alpha_{K_l}^{(j_l)} - \alpha_{K_{l+1}}^{(j_{l+1})}|^2\! +\! |\alpha_{K_{|\omega_\nu|}}^{(j_{|\omega_\nu|})}\! |^2 \mbox{ where we choose } \mu_{d-1}(K_{|\omega_\nu|} \cap \Gamma_D)\! \ne\! 0.\\
    & \mbox{Writing } \omega_\nu ^D = \{ K\in \omega _\nu  \colon \mu_{d-1}(K\cap \Gamma_D) \ne 0\}, \\
    &\forall \nu \in \mathcal{N}_D \; \forall K \in \omega_\nu^D \hspace{5mm} x_K^{(j)}=\nu \implies \alpha_K^{(j)} \mbox{ is the jump over an edge of } \Gamma_D.\\
    & \mbox{Then } B \lesssim \underset{\nu \in \mathcal{N}_D}{\sum} \hspace{1mm} a_1^{\omega_\nu} (\underset{l=1}{\overset{|\omega_\nu|-1}{\sum}} |\alpha_{K_l}^{(j_l)} - \alpha_{K_{l+1}}^{(j_{l+1})}|^2 + \underset{\{(j,K) \colon K \in \omega_\nu^D, x^{(j)}_K=\nu\}}{\sum} |\alpha_K^{(j)}|^2)\\
    &\mbox{Finally, as for any } E \in \mathcal{E}^{\mbox{\tiny int}}\! : \underset{\nu \in E}{\sum} a_1^{\omega_\nu} |\alpha_{K^+}^{(j^+_\nu)}-\alpha_{K^-}^{(j^-_\nu)}|^2  \lesssim a_1^{\omega_E} | [\![  \hat{v}_h  ]\!] |^2_{L^\infty(E)}  \lesssim \frac{a_1^{\omega_E}}{h_E}\! \int_E  [\![  \hat{v}_h  ]\!] ^2\\
    &\mbox{and } \underset{\nu \in \mathcal{N}_D}{\sum} a_1^{\omega_\nu} \hspace{-0.5cm} \underset{\{(j,K) \colon K \in \omega_\nu^D, x^{(j)}_K=\nu\}}{\sum} \hspace{-1cm} |\alpha_K^{(j)}|^2  = \underset{E \in \mathcal{E}_D^{\mbox{\tiny ext}}}{\sum} \underset{\nu \in E}{\sum}a_1^{\omega_\nu} |\alpha_{K_E}^{(j_\nu)}|^2 \lesssim \underset{E \in \mathcal{E}_D^{\mbox{\tiny ext}}}{\sum} \frac{a_1^{\omega_E}}{h_E} \int_E  [\![  \hat{v}_h ]\!] ^2.\\
    &\mbox{ This concludes the proof.}
\end{align*}

\end{proof}

\begin{remark}\label{rqAh}
    Note that the operator $\mathcal{A}_h$, that is used in the proof of lemma \ref{lem:jump}, is a linear operator only depending on the mesh and not on the features. Since the computational mesh is considered static, the operator is independent of time. This result is used in lemma \ref{lem:jump}.
\end{remark}

Let's now project $u_h$ via $\mathcal{A}_h$
\begin{equation}\label{contdisc}
    \hat{u}_h=\hat{u}_h^c + \hat{u}_h^r \mbox{ with } \hat{u}_h^c = \mathcal{A}_h \hat{u}_h
\end{equation}

\noindent Here, $\hat{u}_h^c$ is a continuous projection of $\hat{u}_h$ and $\hat{u}_h^r$ catches the jumps. With this we will find a bound for

\begin{equation} \label{kappa}
    \kappa= |||\hat{u}^s-\hat{u}_h||| + |\hat{u}^s-\hat{u}_h|_A
\end{equation}

\noindent By definition and triangular inequality $\kappa \leq |||\hat{u}^s-\hat{u}_h^c||| + |\hat{u}^s-\hat{u}_h^c|_A + |||\hat{u}_h^r||| + |\hat{u}_h^r|_A$.

\noindent We then bound $\kappa$ with the error estimators. We first bound the jump term by applying lemma \ref{lem3} to $\hat{u}_h^r$. 

\begin{lemma}\label{lem4}

$|||\hat{u}_h^r||| + |\hat{u}_h^r|_A \lesssim (\underset{K \in \mathcal{T}_h}{\sum} [\frac{1}{\alpha} + 1] \eta^{t\;2}_{J_K})^\frac{1}{2}$.
\end{lemma}
\begin{proof}
\noindent Knowing that $ [\![  \hat{u}_h^r  ]\!] = [\![ \hat{u}_h  ]\!] $ on $\mathcal{E}_D$:

\begin{align*}
    &|||\hat{u}_h^r|||^2 + |\hat{u}_h^r|^2_A \leq \underset{K\in \mathcal{T}_h}{\sum} [\varepsilon |\hat{u}_h^r|^2_{U_K} + \beta  \lVert \hat{u}_h^r \rVert ^2_{H_K}] + |(\textbf{V}-\Tilde{\textbf{V}})\hat{u}_h^r|^2_{*} \\
    &+ \underset{E\in \mathcal{E}_D}{\sum}(\beta + \frac{\delta_\infty^{\omega_E}}{\varepsilon}) h_E J_1^{\omega_E} \int_E [\![ \hat{u}_h ]\!]^2 + \frac{\varepsilon \alpha }{h_E} \int_E  J(\textbf{F}^{-T} [\![ \hat{u}_h ]\!] )^2
\end{align*}

\begin{align*}
    &\mbox{By lemma \ref{lem3} }\left \{ \begin{array}{ll}
     \underset{K \in \mathcal{T}_h}{\sum} \varepsilon |\hat{u}_h^r|^2_{U_K} & \lesssim \alpha^{-1} \underset{E \in \mathcal{E}_D}{\sum}  a_1^{\omega_E}\frac{\alpha }{h_E} \int_E  [\![  \hat{u}_h ]\!] ^2 \lesssim \alpha^{-1} \underset{K \in \mathcal{T}_h}{\sum}  \eta^{t\;2}_{J_K}
     \\
     \underset{K \in \mathcal{T}_h}{\sum} \beta  \lVert \hat{u}_h^r \rVert ^2_{H_K} & \lesssim \underset{E \in \mathcal{E}_D}{\sum} \beta h_E J_1^{\omega_E} \int_E  [\![  \hat{u}_h ]\!] ^2 \lesssim \underset{K \in \mathcal{T}_h}{\sum}  \eta^{t\;2}_{J_K}
\end{array}
\right.
\end{align*}

\begin{align*}
    \mbox{And: } |(\textbf{V}-\Tilde{\textbf{V}})\hat{u}_h^r|^2_* & \leq \underset{\hat{v} \in H^1_D(\hat{\Omega})\colon |||\hat{v}|||=1}{\mbox{sup}} ( \lVert  (\textbf{V}-\Tilde{\textbf{V}})\hat{u}_h^r \rVert ^2_{H} \cdot |\hat{v}|_{U}^2) \leq \frac{1}{\varepsilon}  \lVert (\textbf{V}-\Tilde{\textbf{V}})\hat{u}_h^r \rVert ^2_{H(t)} \\
    & \leq \frac{1}{\varepsilon} \underset{K\in \mathcal{T}_h}{\sum} \delta_\infty^K  \lVert \hat{u}_h^r \rVert ^2_{H_K} \lesssim \underset{E\in \mathcal{E}_D}{\sum} \frac{h_E\delta_\infty^{\omega_E}}{\varepsilon} J_1^{\omega_E} \int_E  [\![  \hat{u}_h ]\!] ^2  \lesssim \underset{K \in \mathcal{T}_h}{\sum}  \eta^{t\;2}_{J_K}
\end{align*}
\end{proof}


\noindent Finally, we want to use the inf-sup condition to bound the continuous part of $\kappa$, namely $|||\hat{u}^s-\hat{u}_h^c||| + |\hat{u}^s-\hat{u}_h^c|_A $. To do so we present a lemma of approximation of function of $H^1_D(\hat{\Omega})$ by continuous, piecewise polynomials. This is done for static meshes in \cite{verfurth05} with Clement-type interpolant and I will do it as well here.

\noindent We denote $\mathcal{N}_h$ the vertices of the mesh and $\mathcal{N}_N$ the ones not lying on the Dirichlet boundary and define a nodal basis function $\lambda_y$ for $y \in \mathcal{N}_N$

\begin{align*}
    &\lambda_{y|K}\in \mathcal{P}_1(K) \; \forall K \in \mathcal{T}_h, \; \forall z \in \mathcal{N}_h-\{y\} \; \lambda_y(z)=0 \mbox{ and } \lambda_y(y)=1.\\
    &\mbox{And }\mathcal{I}_h \colon L^1(\hat{\Omega})  \to \{ \varphi \in C(\hat{\Omega}) \colon \varphi|_K \in S_1(K), \varphi =0 \mbox{ on } \Gamma_D \}, \; \hat{v}  \mapsto \underset{y \in \mathcal{N}_N}{\sum}  \lambda_y  \frac{1}{\omega_y^1 }\int_{\omega_y}  \hat{v} \\
\end{align*}

\begin{lemma}\label{lem5}

For all $\hat{v}\in H_D^1(\hat{\Omega})$, there is $|\mathcal{I}_h \hat{v}|_{U_K}^2 \lesssim a_1^{K} a_\infty^{\omega_K} |\hat{v}|_{U_{\omega_K}}^2$ and

\begin{align}
    \underset{K\in \mathcal{T}_h}{\sum} \rho_K^{-2}  \lVert \hat{v}- \mathcal{I}_h \hat{v} \rVert ^2_{L^2(K)} & \lesssim |||\hat{v}|||^2 \label{I1}\\
    \underset{E\in \mathcal{E}}{\sum}\sqrt{ \frac{\varepsilon}{a_\infty^{\omega_E}}}\rho_E^{-1} \int_E (\hat{v}- \mathcal{I}_h \hat{v})^2 & \lesssim |||\hat{v}|||^2 \label{I2}
\end{align}

\end{lemma}

\begin{proof}

Since $\mathcal{I}_h \hat{v}$ is piecewise affine, we can write $\nabla_X \mathcal{I}_h \hat{v} = n_{\hat{v}} ^K$. Then, \cite{verfurth05}, \textbf{lemma 5.1} states that $|K| \cdot  \lVert n_{\hat{v}}^K \rVert ^2 \lesssim  \lVert \nabla _X\hat{v} \rVert _{L^2(\omega_K)}^2$. This proves the first inequality.

\begin{equation*}
    \mbox{\textbf{Lemma 3.1} in \cite{verfurth982} yields } \left \{ \! \begin{array}{ll}
          \frac{\beta}{M_\infty^{\omega_K}} \! \int_K \! (\hat{v}- \mathcal{I}_h \hat{v})^2 & \lesssim \! \frac{\beta}{M_\infty^{\omega_K}} \! \lVert \hat{v} \rVert ^2_{L^2(\omega_K)} \! \lesssim \! \beta  \lVert \hat{v} \rVert ^2_{H_{\omega_K}}\\
         \frac{\varepsilon}{h_K^2 a_\infty^{\omega_K}} \! \int_K \! (\hat{v}- \mathcal{I}_h \hat{v})^2 & \lesssim \! \frac{\varepsilon}{a_\infty^{\omega_K}} \!  \lVert \nabla_X \hat{v} \rVert ^2_{L^2(\omega_K)} \! \lesssim \! \varepsilon |\hat{v}|^2_{U_{\omega_K}}
\end{array}
\right .\\
\end{equation*}

\noindent which proves \eqref{I1}. \eqref{I2} follows with the technique in \textbf{lemma 3.2} in \cite{verfurth982}.
\end{proof}

\noindent We can bound $|||\hat{u}^s-\hat{u}_h^c||| + |\hat{u}^s-\hat{u}_h^c|_A$ with the inf-sup property and the previous lemmas.

\begin{lemma}\label{lem6}
$|||\hat{u}^s\!-\!\hat{u}_h^c|||^2 \!+\! |\hat{u}^s\!-\!\hat{u}_h^c|_A^2 \lesssim \underset{K \in \mathcal{T}_h}{\sum} (1+\frac{1}{\alpha}) \eta_K^{t\; 2}$ with $\hat{u}_h^c$ defined in \eqref{contdisc}.
\end{lemma}

\begin{proof}

Let $\hat{v} \in H^1_D(\hat{\Omega})$, we first bound $T(\hat{v}) = l(\hat{v}-\mathcal{I}_h \hat{v}) - \tilde{a}_h(\hat{u}_h,\hat{v}-\mathcal{I}_h \hat{v})$. $T\!=\!T_1\!+\!T_2\!+\!T_3$ with 

\begin{equation*}
    \left \{  \begin{array}{ll}     T_1(\hat{v}) =  \underset{K \in \mathcal{T}_h}{\sum} \!  \int_K \! J  \!(\!\hat{f}\! -\!\frac{\partial \hat{u}_h}{\partial t}  \! + \! \frac{\varepsilon}{J}  \nabla _X \! \cdot\!  \{J \textbf{F}^{-1} \textbf{F}^{-T} \nabla _X \hat{u}_h\} \! -\! (\textbf{V}\! -\! \Tilde{\textbf{V}})\! \cdot\!  \textbf{F}^{-T}\nabla _X \hat{u}_h)(\hat{v}\! -\! \mathcal{I}_h \hat{v})  \\
     T_2(\hat{v}) =    - \varepsilon \underset{K \in \mathcal{T}_h}{\sum}  \int_{\partial K} (J\textbf{F}^{-1} \textbf{F}^{-T} \nabla_X \hat{u}_h\cdot \textbf{n}_K)(\hat{v}-\mathcal{I}_h \hat{v})\\
     T_3(\hat{v}) =  -\underset{K \in \mathcal{T}_h}{\sum}  \int_{\partial K_{\mbox{\tiny in}}^t-\Gamma_r} (J(\textbf{V}-\Tilde{\textbf{V}}) \cdot \textbf{F}^{-T} [\![  \hat{u}_h  ]\!]) (\hat{v}-\mathcal{I}_h \hat{v})
\end{array}
\right.
\end{equation*}

\noindent And by Cauchy-Schwarz and lemma \ref{lem5}

\begin{multline} \label{Ts}
    |T_1|  \! \lesssim \! (\!\underset{K\in \mathcal{T}_h}{\sum} \eta_{R_K}^{t\; 2}\!)^\frac{1}{2} |||\hat{v}|||, \; |T_2| \! \lesssim \! (\!\underset{K \in \mathcal{T}_h}{\sum} \eta_{E_K}^{t\; 2}\!)^\frac{1}{2} |||\hat{v}||| \mbox{ and } |T_3| \! \lesssim \!  (\!\underset{K \in \mathcal{T}_h}{\sum} \eta_{J_K}^{t\; 2}\!)^\frac{1}{2} |||\hat{v}||| 
\end{multline}

\noindent There is  $l(\mathcal{I}_h \hat{v})\!  =\!\tilde{a}_h(\hat{u}_h,\mathcal{I}_h \hat{v})\!+\!\Tilde{p}_h(\hat{u}_h,\mathcal{I}_h \hat{v})$ thus 

\begin{center}
    $\tilde{a}_h(\hat{u}^s\! -\!\hat{u}_h^c,\hat{v}) \!= \! l(\hat{v}) \! -\! \tilde{a}_h(\hat{u}_h,\hat{v})\! +\! \tilde{a}_h(\hat{u}_h^r,\hat{v})\!=\! T(\hat{v})\!+\! \tilde{a}_h(\hat{u}_h^r,\! \hat{v})\! +\! \Tilde{p}_h(\hat{u}_h,\!\mathcal{I}_h \hat{v})$.
\end{center} 

\noindent Therefore by \eqref{Ts} and lemmas \ref{lem:continuity}, \ref{lem2}, \ref{lem4}, \ref{lem5}

\begin{equation}\label{aheta}
    |\tilde{a}_h(\hat{u}^s-\hat{u}_h^c,\hat{v})| \lesssim \{(\underset{K\in \mathcal{T}_h}{\sum} \eta_{K}^{t\; 2})^\frac{1}{2} + (\underset{K \in \mathcal{T}_h}{\sum} [\frac{1}{\alpha} + 1] \eta_{J_K}^{t\; 2})^\frac{1}{2} +(\underset{K\in \mathcal{T}_h}{\sum} \frac{1}{\alpha} \eta_{J_K}^{t\; 2})^{\frac{1}{2}}\} |||\hat{v}|||. \\
\end{equation}

\noindent And by noting that $|\hat{u}^s-\hat{u}_h^c|_A= |(\textbf{V}-\Tilde{\textbf{V}})(\hat{u}^s-\hat{u}_h^c)|_*$, we can use lemma \ref{lem:infsup}

\begin{multline*}
    |||\hat{u}^s-\hat{u}_h^c||| + |\hat{u}^s-\hat{u}_h^c|_A \lesssim \underset{\hat{v}\in H_D^1(\hat{\Omega})-\{0\}}{\sup}\frac{\tilde{a}_h(\hat{u}^s-\hat{u}_h^c,\hat{v})}{|||\hat{v}|||} \mbox{ and conclude with \eqref{aheta}.} 
\end{multline*}

\end{proof}

\noindent Finally we state theorem \ref{thm:apststd} where $\kappa$ holds the role of $\hat{\sigma}$.

\begin{lemma}\label{lemthm1}
$\kappa^2 \lesssim \underset{K \in \mathcal{T}_h}{\sum} (1+\frac{1}{\alpha}) \eta_K^{t\; 2}$ with $\kappa$ defined in \eqref{kappa}.
\end{lemma}

\begin{proof}
This result follows directly from lemmas \ref{lem4} and \ref{lem6}.
\end{proof}

\noindent By splitting the solution into a continuous and a discontinuous part (lemma \ref{lem3}) and bounding these two contributions we could use the inf-sup condtion and conclude with lemma \ref{lemthm1}.

\noindent This concludes the proof of the reliability of the steady-state criterion and thus the proof of the space-time criteria. 

\noindent In the next section we will look at the value of $\eta_1^t$ (which is the space-time equivalent of $\eta_K$) with a boundary-layer test case. We show in particular that the criterion is still able to represent the error when the mesh moves and that the moving mesh situation improves the solution.

\section{A boundary layer test case} \label{sec:testcase}

As presented in the introduction, the robustness of the DG method and its \textit{a posteriori} error criteria is demonstrated by its ability to resolve boundary layers stably even when the mesh Peclet number becomes large. In \cite{schotzau09} they demonstrate that the \textit{a posteriori} refinement criterion only works in this condition, the mesh Peclet number decreases, it is not efficient anymore. This section will present a boundary layer test case and compare the solution on a static mesh with that on a moving mesh.

The solution is: $u(t,x,y)\! =\! (1-e^{-t})[\frac{e^{(x-1)/\varepsilon}-1}{e^{-1/\varepsilon}-1}\! +\!x\! -\!1][\frac{e^{(y-1)/\varepsilon}-1}{e^{-1/\varepsilon}-1} \!+\! y\! -\! 1]$ in the square $[0,1]^2$. It has a boundary layer of size $o(\varepsilon)$ on the upper and right boundaries. The advection velocity is $\textbf{V}(t,x,y)\! =\! (1,1)^T\! +\! \Tilde{\textbf{V}}(t,x,y)$ where $\Tilde{\textbf{V}}$ is a divergence-free velocity such that $\Tilde{\textbf{V}}(t,x,y)\! =\! 2^{16} (h(y)h'(x),\! -h(x)h'(y))^T$ with ${h(x)\! =\! (x(1-x))^2}$ and the right hand side is computed correspondingly. In this case, the remaining advection is divergence-free for both the static and moving mesh simulations. $\Tilde{\textbf{V}}$ advects the flow on the closed curves as shown in figure \ref{fig:1} with a maximal velocity along the black solid line.

\begin{center}
    \begin{figure}[h]
        \centering
        \includegraphics[trim=0cm 1cm 0cm 1cm,scale=0.3]{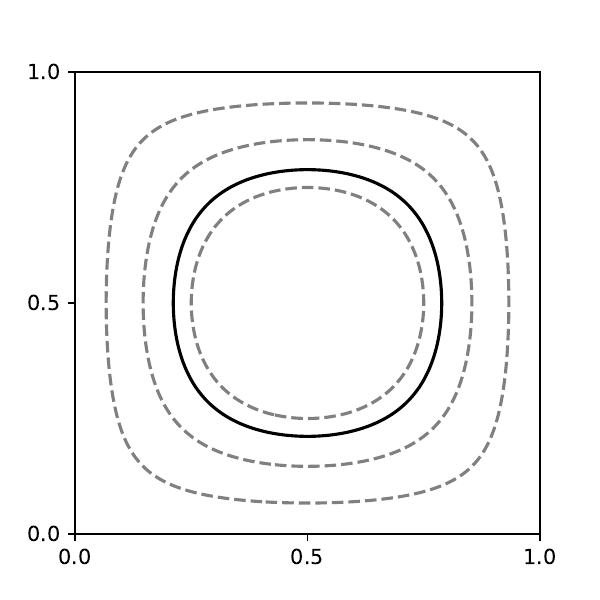}
        \caption{Curves on which the particles are advected by the velocity $\Tilde{\textbf{V}}$}
        \label{fig:1}
    \end{figure}
\end{center}

We choose to execute the simulation on a regular triangular mesh of size $h \approx \frac{1}{9}$ with $\varepsilon=\frac{1}{100}$. For the time-stepping we use an explicit Runge-Kutta scheme of order 4 with timestep $\Delta t=2^{-16}$ for the DG semi-discretisation and an explicit RK4 with timestep $\frac{\Delta t}{2}$ for the ODE defining the flowmap. In figure \ref{fig:2} we plot the error and the spatial criteria respectively after one and twelve timesteps for the moving mesh and static mesh simulation.

\begin{figure}%
\centering
\includegraphics[trim=7.5cm 1.5cm 8cm 1cm,scale=0.49]{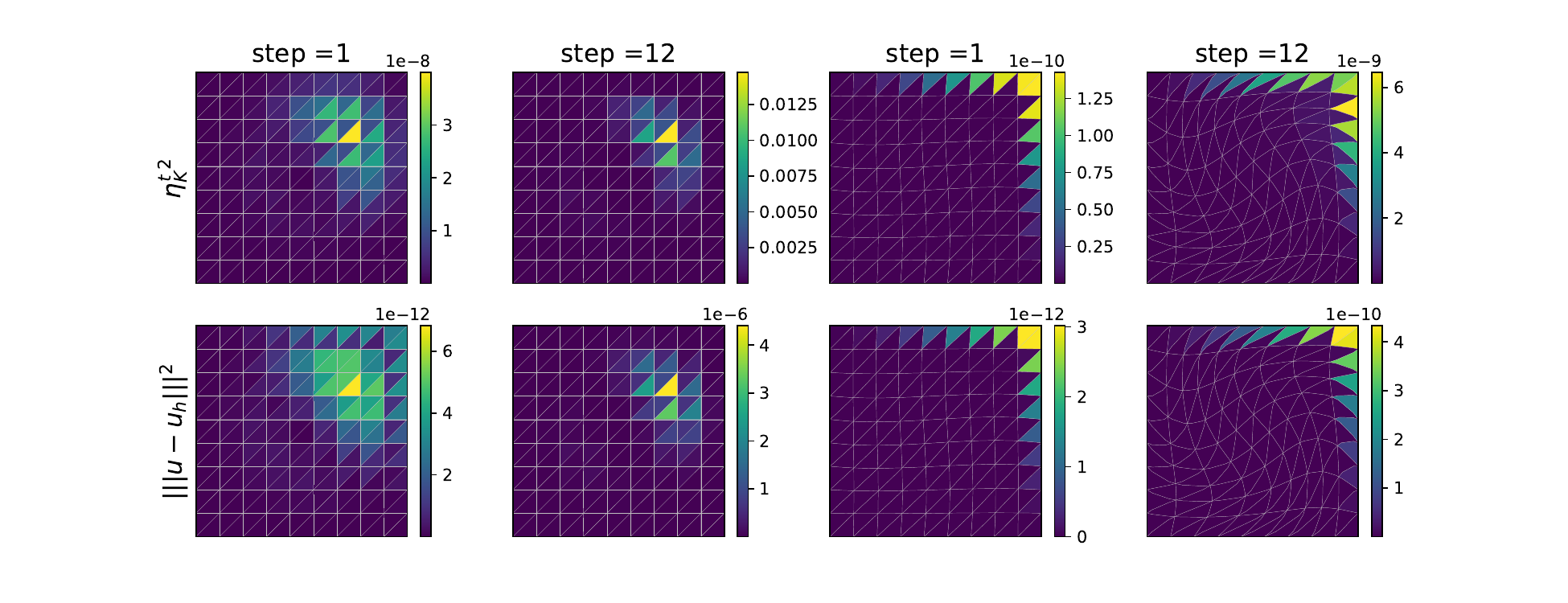}
\caption{Plot of  the criterion (upper row) and the error (lower row) after one and after twelve time steps for a static mesh (left) and a moving mesh (right). Note the differences in the scaling, which is left for better readability.}
\label{fig:2}
\end{figure}

In both of these simulations we can see that the error criteria can accurately catch the error. In the static mesh simulation the error is dominated by the advective part, this makes the boundary layer problem less relevant and the refinement would focus on the area where the advection velocity is maximal (black curve in figure \ref{fig:1}) and would not resolve smaller scale effects like the boundary layer. In the moving mesh situation we can see that the simulation is more accurate everywhere and in particular in the center of the square and the area with maximal error is on the boundary. Consequently, the mesh refinement would focus on the boundary layer and resolve this small scale effect.

In the moving mesh simulation, we can also see that for longer times, there would be a point where entanglement of the mesh occurs. Since theoretically $\chi(t,\cdot)\in C^2(\hat{\Omega},\Omega)$, the trajectories are not crossing, therefore entanglement can be avoided by improving the resolution of the moving mesh's ODE (for instance spatially more accurate with more computed trajectories). This is independent of the DG method and thus not treated here.

\section{Conclusion}\label{sec:conc}

In this work we established interior penalty discontinuous Galerkin methods for the semi-discretisation of an Arbitrary Lagrangian-Eulerian formulation in unsteady advection-diffusion problems. By discretising the problem via a dynamically deforming map, we used the existing analytic techniques for advection-diffusion problems with continuous diffusion tensors. This lead us to derive \textit{a priori} error estimates and made the establishment of \textit{a posteriori} error criteria possible. The reliability of the \textit{a posteriori} error estimations were then discussed in a numerical test.

The \textit{a priori} error estimate shows a condition on the moving mesh velocity for the spatial convergence. This results in a second order convergence in space when the polynomial order is larger than 1, the mesh Peclet number is large and the remaining advection velocity is smaller than the diffusion term. This is a higher order than the one presented in \cite{dolesji15}.

By focusing on the available data, we derived specific \textit{a posteriori} criteria for the moving mesh method in two spatial dimensions. The robustness of these error criteria in terms of the mesh Peclet number allowed us to scale the error criteria with the square of the local remaining advection speed (called $\delta$). This behaviour is confirmed by the test cases where the error in terms of the energy-norm appears to strongly depend on this local speed. 

In section \ref{sec:testcase} we showed that this moving mesh method inherits from the robustness properties of the DG method on static meshes. Similarly, the \textit{a posteriori} error criteria are able to robustly represent the error.

\backmatter

\bmhead{Supplementary information}

This article is accompanied with a proof of what is stated in remark \ref{rq:velocity}.

\section*{Declarations}

\textbf{Competing Interests} The authors have no competing interests to declare that are relevant to the content of this article.

\begin{appendices}

\section{Proof of remark \ref{rq:velocity}}\label{secA1}

In this appendix we treat the question of the proof of remark \ref{rq:velocity} in the paper. The subject of the remark is to show that $\Tilde{\textbf{V}}(t,\cdot)\in W^{3,\infty}(\Omega)^2$ is sufficient so $\forall \textbf{q} \in \mathbb{R}^2, \; J^\frac{1}{2}\textbf{F}^{-T}\textbf{q}\in H^2(\mathcal{T}_h)^2$. Aditionnally we also give a value to the following exponential term $C_2(t)$. Suppose $\Tilde{\textbf{V}}\in W^{3,\infty}(\Omega)$, let

\begin{align} \label{apdx:defs}
    \textbf{B}(t,x)=& \frac{\nabla_x \cdot \Tilde{\textbf{V}}}{2} \mbox{\textbf{Id}} - (\nabla_x \Tilde{\textbf{V}})^T \nonumber\\
    C_0(t) =&  2||\textbf{B}||_{(1,0)} \nonumber\\
    C_1(t)=&4||\textbf{B}||_{(1,0)}+||\textbf{B}||_{(1,1)} + 2||\, ||D(\Tilde{\textbf{V}})||_2 \, ||_{(1,0)} + \ln(2)\\
    C_2(t) =& 6||\textbf{B}||_{(1,0)}+[5 +2\sqrt{2}\,||\,||\nabla_x D(\Tilde{\textbf{V}})||_2\,||_{(1,0)}]||\textbf{B}||_{(1,1)}\nonumber\\
    &+2\sqrt{2} ||\textbf{B}||_{(1,2)} + 4\sqrt{2}\, ||\, ||D(\Tilde{\textbf{V}})||_2 \, ||_{(1,0)}\nonumber
\end{align}

\noindent We define some norms on matrices, for $\textbf{A}\in \mathcal{M}^{n}(\mathbb{R})$ we write $|||\textbf{A}|||$ the operator norm and $||\textbf{A}||_2^2=\underset{i,j}{\sum} (A_{ij})^2$.

\noindent For $\textbf{A} \colon \Omega \rightarrow \mathcal{M}^{n}(\mathbb{R})$ we define the partial derivative in the reference variable $X$ of order $\alpha$ of $\textbf{A}$ as $\partial_\alpha \textbf{A} = (\partial_\alpha A_{ij})_{i,j}$, if $|||\textbf{A}||| \in L^\infty (\Omega)$ we write $|||\textbf{A}|||_\infty = \underset{x\in \Omega}{\max}|||\textbf{A}(x)|||$ and for $p\in \mathbb{N}$, $|\textbf{A}|_{W^{p,\infty}(\Omega)}=\underset{|\alpha|=p}{\max}|||\partial_\alpha \textbf{A}|||_\infty$. We also write $\partial_{x_j}$ the partial derivative in the spatial variable $x$. We recall that we work with 2 spatial dimensions.

\noindent Finally for $\textbf{A} \colon [0,t] \times \Omega \rightarrow \mathcal{M}^{n}(\mathbb{R})$ and for $p,q \in \mathbb{N}$ let $||\textbf{A}||_{(p,q)}=(\int_0^t |\textbf{A}|_{W^{q,\infty}}^p)^\frac{1}{p}$.

\noindent We will prove the following lemma that implies remark 4.2.

\begin{lemma} \label{apdx:lmvelo}
Let $\textbf{B},C_0,C_1,C_2$ defined in \eqref{apdx:defs}, let $(\textbf{q}_K)_{K\in\mathcal{T}_h}$ a collection of vectors such that $||\textbf{q}_K||=1$. Finally we write $\textbf{Y}$ such that $\textbf{Y}|_K=J^\frac{1}{2}\textbf{F}^{-T}\textbf{q}_K$. Then:

\begin{align}
    ||\textbf{Y}||_{L^2(K)}^2 &\leq|K| \exp(C_0(t))\label{apdx:Y1}\\ 
    |\textbf{Y}|_{H^1(K)}^2 &\leq|K| \exp(C_1(t))\label{apdx:Y2}\\ 
    |\textbf{Y}|_{H^2(K)}^2  &\leq|K| \exp(C_2(t))\label{apdx:Y3}
\end{align}

\end{lemma}

\begin{proof}
    We first notice that since $\Tilde{\textbf{V}}\in W^{3,\infty}(\Omega)$ then $\textbf{B}\in W^{2,\infty}(\Omega)$.
    
\noindent \textbf{Proof of \eqref{apdx:Y1}:} by noticing that $\dot{\textbf{Y}}=\textbf{B}\textbf{Y}$

$\frac{d}{dt}||\textbf{Y}||_{L^2(K)^2}^2\leq 2|||\textbf{B}|||_\infty \cdot ||\textbf{Y}||_{L^2(K)}^2$

\noindent And by Gronwall's inequality: 
\begin{equation*}
    ||\textbf{Y}||_{L^2(K)}^2 \leq|K| \exp(2||\textbf{B}||_{(1,0)})
\end{equation*}

\noindent \textbf{Proof of \eqref{apdx:Y2}:} we notice here that 

$\dot{(\partial_i \textbf{Y})}=\partial_i \dot{\textbf{Y}} = (\partial_i \textbf{B})\textbf{Y}+\textbf{B}(\partial_i \textbf{Y})=\textbf{B}(\partial_i \textbf{Y})+\textbf{Y} \underset{j=1,2}{\sum}\partial_i x_j \, \partial_{x_j}\textbf{B} $.

\noindent Let $\textbf{A}_1,\textbf{A}_2\in \mathcal{M}^{2}(\mathbb{R})$ define as:

$\textbf{A}_1 = \begin{pmatrix}
  ((\partial_{x_1}\textbf{B})\textbf{Y})^T \\
  ((\partial_{x_2}\textbf{B})\textbf{Y})^T
\end{pmatrix}$ and $\textbf{A}_2 = (\partial_1 \textbf{Y} \; \partial_2 \textbf{Y})$.

Then $\frac{d}{dt}|\textbf{Y}|_{H^1(K)}^2=2\int_K \mbox{Tr}(\textbf{F}^T\textbf{A}_1\textbf{A}_2)+\textbf{B}\partial_1 \textbf{Y} \cdot \partial_1 \textbf{Y}+\textbf{B}\partial_2 \textbf{Y} \cdot \partial_2 \textbf{Y}$
\begin{align*}
\frac{d}{dt}|\textbf{Y}|_{H^1(K)}^2&=2\int_K \mbox{Tr}(\textbf{F}^T\textbf{A}_1\textbf{A}_2)+\textbf{B}\partial_1 \textbf{Y} \cdot \partial_1 \textbf{Y}+\textbf{B}\partial_2 \textbf{Y} \cdot \partial_2 \textbf{Y}\\
    & \leq 2(\int_K \mbox{Tr}(\textbf{A}_1^T\textbf{A}_1))^\frac{1}{2} (\int_K \mbox{Tr}(\textbf{A}_2\textbf{F}^T\textbf{F}\textbf{A}_2^T))^\frac{1}{2}+2|||\textbf{B}|||_\infty \cdot |\textbf{Y}|_{H^1(K)}^2
\end{align*}

Now since for all $\textbf{G}\in\mathcal{M}^2(\mathbb{R})$:
\begin{align*}
    \frac{d}{dt} \mbox{Tr}(\textbf{G}\textbf{F}^T\textbf{F}\textbf{G}^T)&=2\mbox{Tr}(\textbf{G}\textbf{F}^TD(\Tilde{\textbf{V}})\textbf{F}\textbf{G}^T)\\ 
    & \leq 2 \mbox{Tr}(D(\Tilde{\textbf{V}})^2)^\frac{1}{2}\mbox{Tr}((\textbf{G}\textbf{F}^T\textbf{F}\textbf{G}^T)^2)^\frac{1}{2}  \\
    &\leq  2 \mbox{Tr}(D(\Tilde{\textbf{V}})^2)^\frac{1}{2}\mbox{Tr}(\textbf{G}\textbf{F}^T\textbf{F}\textbf{G}^T)
\end{align*}

\noindent Where the last inequality holds since $\textbf{G}\textbf{F}^T\textbf{F}\textbf{G}^T$ symmetric positive.

By Gronwall's inequality:
\begin{equation}
    \mbox{Tr}(\textbf{G}\textbf{F}^T\textbf{F}\textbf{G}^T)\leq e^{2||\,||D(\Tilde{\textbf{V}})||_2\,||_{(1,0)}} \mbox{Tr}(\textbf{G}\textbf{G}^T) \label{apdx:trace}
\end{equation}

\noindent Moreover $\left \{ \begin{array}{ll}
\int_K \mbox(\textbf{A}_2\textbf{A}_2^T) = |\textbf{Y}|_{H^1(K)}^2 \\
\int_K \mbox(\textbf{A}_1\textbf{A}_1^T) \leq 2||\textbf{B}||_{W^{1,\infty}}^2||\textbf{Y}||_{L^2(K)}^2
\end{array}
\right.$

\noindent Finally 
\begin{align*}
    \frac{d}{dt}|\textbf{Y}|_{H^1(K)}^2\leq(2||\textbf{B}||_{W^{0,\infty}}&+||\textbf{B}||_{W^{1,\infty}})|\textbf{Y}|_{H^1(K)}^2\\
    &+2||\textbf{B}||_{W^{1,\infty}}e^{2||\,||D(\Tilde{\textbf{V}})||_2\,||_{(1,0)}}||\textbf{Y}||_{L^2(K)}^2
\end{align*}

\noindent And Gronwall's inequality gives \eqref{apdx:Y2}:
\begin{equation*}
    |\textbf{Y}|_{H^1(K)}^2 \leq 2|K|\exp(4||\textbf{B}||_{(1,0)}+||\textbf{B}||_{(1,1)} + 2||\, ||D(\Tilde{\textbf{V}})||_2 \, ||_{(1,0)})
\end{equation*}

\noindent \textbf{Proof of \eqref{apdx:Y3}:} we notice here that 

$\dot{(\partial_{ij} \textbf{Y})}=(\partial_{ij}\textbf{B})\textbf{Y} + (\partial_i\textbf{B})(\partial_j \textbf{Y})+ (\partial_j\textbf{B})(\partial_i \textbf{Y}) + \textbf{B}(\partial_{ij}\textbf{Y})$

\noindent Then:

\noindent $\frac{d}{dt}|\textbf{Y}|_{H^2(K)}^2=2 \underset{i,j=1,2}{\sum}\int_K[(\partial_{ij}\textbf{B})\textbf{Y} + (\partial_i\textbf{B})(\partial_j \textbf{Y})+ (\partial_j\textbf{B})(\partial_i \textbf{Y}) + \textbf{B}(\partial_{ij}\textbf{Y})]\cdot \partial_{ij}\textbf{Y}$

Let $\textbf{A}_{1j} = \begin{pmatrix}
  ((\partial_{x_1}\textbf{B})\partial_j\textbf{Y})^T \\
  ((\partial_{x_2}\textbf{B})\partial_j\textbf{Y})^T
\end{pmatrix}$ and $\textbf{A}_{2j} = (\partial_1 \partial_j \textbf{Y} \; \partial_2 \partial_j \textbf{Y})$.

\noindent Then 
\begin{align*}
    \underset{i=1,2}{\sum}\int_K (\partial_i\textbf{B})(\partial_j \textbf{Y})\cdot \partial_{ij}\textbf{Y}&=\int_K \mbox{Tr}(\textbf{F}^T\textbf{A}_{1j}\textbf{A}_{2j})\\
    & \leq (\int_K \mbox{Tr}(\textbf{A}_{1j}^T\textbf{A}_{1j}))^\frac{1}{2} (\int_K \mbox{Tr}(\textbf{A}_{2j}\textbf{F}^T\textbf{F}\textbf{A}_{2j}^T))^\frac{1}{2}\\
    & \leq \sqrt{2} ||\textbf{B}||_{W^{1,\infty}}e^{||\,||D(\Tilde{\textbf{V}})||_2\,||_{(1,0)}}||\partial_j\textbf{Y}||_{L^2(K)}|\partial_j\textbf{Y}|_{H^1(K)} \\
    & \leq ||\textbf{B}||_{W^{1,\infty}}e^{2||\,||D(\Tilde{\textbf{V}})||_2\,||_{(1,0)}}||\partial_j\textbf{Y}||_{L^2(K)}^2\\
    &\hspace{1cm} + \frac{||\textbf{B}||_{W^{1,\infty}}}{2}|\partial_j\textbf{Y}|_{H^1(K)}^2
\end{align*}

\noindent Thus 
\begin{align*}
    &\frac{d}{dt}|\textbf{Y}|_{H^2(K)}^2 \leq (2||\textbf{B}||_{W^{0,\infty}}+4||\textbf{B}||_{W^{1,\infty}})|\textbf{Y}|_{H^2(K)}^2\\
    &+2||\textbf{B}||_{W^{1,\infty}}e^{2||\,||D(\Tilde{\textbf{V}})||_2\,||_{(1,0)}}|\textbf{Y}|_{H^1(K)}^2+2 \underset{i,j=1,2}{\sum}\int_K(\partial_{ij}\textbf{B})\textbf{Y}\cdot \partial_{ij}\textbf{Y}
\end{align*} 

We finally bound the last sum of the previous inequality:
\begin{align*}
\underset{i,j=1,2}{\sum}\int_K(\partial_{ij}\textbf{B})\textbf{Y}\cdot \partial_{ij}\textbf{Y} =& \underset{i,j=1,2}{\sum}\int_K \partial_i\{\partial_jx_1\partial_{x_1}\textbf{B}+\partial_jx_2\partial_{x_2}\textbf{B}\}\textbf{Y}\cdot \partial{ij}\textbf{Y}\\
\textbf{S}_1 \leftarrow \hspace{1cm} =&\underset{i,j=1,2}{\sum}\int_K \{\partial_{ij}x_1\partial_{x_1}\textbf{B}+\partial_{ij}x_2\partial_{x_2}\textbf{B}\}\textbf{Y}\cdot \partial{ij}\textbf{Y}\\
\textbf{S}_2 \leftarrow \hspace{1.3cm} &+\underset{i,j=1,2}{\sum}\int_K\{\partial_{i}x_1\partial_{j}x_1\partial_{x_1}^2\textbf{B}+\partial_{i}x_1\partial_{j}x_2\partial_{x_1x_2}^2\textbf{B}\\
&+\partial_{j}x_1\partial_{i}x_2\partial_{x_1x_2}^2\textbf{B}+\partial_{i}x_2\partial_{j}x_2\partial_{x_2}^2\textbf{B}\}\textbf{Y}\cdot \partial{ij}\textbf{Y}
\end{align*}

\noindent We define $\textbf{A}_{3k}\! =\! \begin{pmatrix}
  ((\partial_{x_jx_1}\textbf{B})\partial_j\textbf{Y})^T \\
  ((\partial_{x_jx_2}\textbf{B})\partial_j\textbf{Y})^T
\end{pmatrix}$
and $\textbf{C}\! =\! (\mbox{Tr}(\textbf{F}^T\textbf{A}_{3j}\textbf{A}_{2i}))_{i,j=1,2}$.\\

\noindent Then $\textbf{S}_2\! =\! \int_K \! \mbox{Tr}(\textbf{C}\textbf{F})$ and:
\begin{align*}
    \textbf{S}_2 &\leq (\int_K \mbox{Tr}(\textbf{C}\textbf{C}^T))^\frac{1}{2}(\int_K \mbox{Tr}(\textbf{F}\textbf{F}^T))^\frac{1}{2} \\
    &\leq \sqrt{2} e^{||\,||D(\Tilde{\textbf{V}})||_2\,||_{(1,0)}}(\int_K \mbox{Tr}(\textbf{C}\textbf{C}^T))^\frac{1}{2}\\
    &\leq \sqrt{2} e^{||\,||D(\Tilde{\textbf{V}})||_2\,||_{(1,0)}}(\int_K \underset{i,j=1,2}{\sum}\mbox{Tr}(\textbf{F}^T\textbf{A}_{3j}\textbf{A}_{2i})^2)^\frac{1}{2}\\
    &\leq 2\sqrt{2} e^{2||\,||D(\Tilde{\textbf{V}})||_2\,||_{(1,0)}}||\textbf{B}||_{W^{2,\infty}}||\textbf{Y}||_{L^2(K)}|\textbf{Y}|_{H^2(K)}
\end{align*}

\noindent Finally for bounding $\textbf{S}_1$ we define the following matrices $\mathcal{F},\textbf{G},\textbf{H},\textbf{C}_1,\textbf{C}_2 \in \mathcal{M}^4(\mathbb{R})$:

 $\mathcal{F}=\begin{pmatrix}
  \partial_1x_1\mbox{\textbf{Id}}&\partial_2x_1\mbox{\textbf{Id}} \\
  \partial_1x_2\mbox{\textbf{Id}}&\partial_2x_2\mbox{\textbf{Id}}
\end{pmatrix}$, $\textbf{G}=\begin{pmatrix}
  \partial_1\textbf{F}&0 \\
  0&\partial_2\textbf{F}
\end{pmatrix}$, $\textbf{H}=\begin{pmatrix}
  \textbf{F}&0 \\
  0&\textbf{F}
\end{pmatrix}$

$\textbf{C}_1=\begin{pmatrix}
  \textbf{A}_1&0 \\
  0&\textbf{A}_1
\end{pmatrix}$, $\textbf{C}_2=\begin{pmatrix}
  \textbf{A}_{21}&0 \\
  0&\textbf{A}_{22}
\end{pmatrix}$

\noindent With these notations we have that:

\begin{equation*}
    \textbf{S}_1=\int_K \mbox{Tr}(\textbf{G}^T\textbf{C}_1\textbf{C}_2)\leq (\int_K \mbox{Tr}(\textbf{C}_1\textbf{C}_1^T))^\frac{1}{2}(\int_K \mbox{Tr}(\textbf{C}_2\textbf{G}^T\textbf{G}\textbf{C}_2))^\frac{1}{2}
\end{equation*}

\noindent For all $\textbf{A}\in \mathcal{M}^4(\mathbb{R})$
\begin{align*}
    \frac{d}{dt}\mbox{Tr}(\textbf{A}\textbf{G}^T\textbf{G}\textbf{A}^T)=&2\mbox{Tr}(\textbf{A}\textbf{G}^T
\begin{pmatrix}
  D(\Tilde{\textbf{V}})&0 \\
  0&D(\Tilde{\textbf{V}})
\end{pmatrix}
\textbf{G}\textbf{A}^T) \\
&+ 2\mbox{Tr}(\textbf{A}\textbf{H}^T
\begin{pmatrix}
  \partial_1D(\Tilde{\textbf{V}})&0 \\
  0&\partial_2D(\Tilde{\textbf{V}})
\end{pmatrix}
\textbf{H}\textbf{A}^T) \\
 \leq & 2\sqrt{2}\mbox{Tr}(D(\Tilde{\textbf{V}})^2)^\frac{1}{2}\mbox{Tr}(\textbf{A}\textbf{G}^T\textbf{G}\textbf{A}^T)\\
&+2\mbox{Tr}((\partial_1D(\Tilde{\textbf{V}}))^2+(\partial_2D(\Tilde{\textbf{V}}))^2)^\frac{1}{2}\mbox{Tr}(\textbf{A}\textbf{H}^T\textbf{H}\textbf{A}^T)
\end{align*}

\noindent And $\begin{pmatrix}
  \partial_1D(\Tilde{\textbf{V}})&0 \\
  0&\partial_2D(\Tilde{\textbf{V}})
\end{pmatrix} = (\mathcal{F}^T\begin{pmatrix}
  \partial_{x_1}D(\Tilde{\textbf{V}}) \\
  \partial_{x_2}D(\Tilde{\textbf{V}})
\end{pmatrix})^T(\mathcal{F}^T\begin{pmatrix}
  \partial_{x_1}D(\Tilde{\textbf{V}}) \\
  \partial_{x_2}D(\Tilde{\textbf{V}})
\end{pmatrix})$
implies that:
\begin{equation*}
    \mbox{Tr}((\partial_1D(\Tilde{\textbf{V}}))^2+(\partial_2D(\Tilde{\textbf{V}}))^2)\leq e^{2\sqrt{2}||\,||D(\Tilde{\textbf{V}})||_2\,||_{(1,0)}}\mbox{Tr}((\partial_{x_1}D(\Tilde{\textbf{V}}))^2+(\partial_{x_2}D(\Tilde{\textbf{V}}))^2)
\end{equation*}

Also: $\mbox{Tr}(\textbf{A}\textbf{H}^T\textbf{H}\textbf{A}^T)\leq e^{2\sqrt{2}||\,||D(\Tilde{\textbf{V}})||_2\,||_{(1,0)}}\mbox{Tr}(\textbf{A}\textbf{A}^T)$.

Thus 
\begin{align*}
    \frac{d}{dt}\mbox{Tr}(\textbf{A}\textbf{G}^T\textbf{G}\textbf{A}^T)&\leq 2\sqrt{2}||D(\Tilde{\textbf{V}})||_2\mbox{Tr}(\textbf{A}\textbf{G}^T\textbf{G}\textbf{A}^T)\\
    &+2e^{2\sqrt{2}||\,||D(\Tilde{\textbf{V}})||_2\,||_{(1,0)}}||\nabla_xD(\Tilde{\textbf{V}})||_2\mbox{Tr}(\textbf{A}\textbf{A}^T)
\end{align*}

And by Gronwall's inequality:
\begin{align*}
    \mbox{Tr}(\textbf{A}\textbf{G}^T\textbf{G}\textbf{A}^T)&\leq 2e^{4\sqrt{2}||\,||D(\Tilde{\textbf{V}})||_2\,||_{(1,0)}}||\,||\nabla_xD(\Tilde{\textbf{V}})||_2\,||_{(1,0)}\mbox{Tr}(\textbf{A}\textbf{A}^T)
\end{align*}

\noindent And since $\int_K\mbox{Tr}(\textbf{C}_1\textbf{C}_1^T)\leq4||\textbf{B}||_{W^{1,\infty}}^2||\textbf{Y}||_{L^2(K)}^2$ and $\int_K\mbox{Tr}(\textbf{C}_2^T\textbf{C}_2)=|\textbf{Y}|_{H^2(K)}^2$:
\begin{equation*}
    \textbf{S}_1\leq 2\sqrt{2}||\textbf{B}||_{W^{1,\infty}}||\,||\nabla_xD(\Tilde{\textbf{V}})||_2\,||_{(1,0)}e^{2\sqrt{2}||\,||D(\Tilde{\textbf{V}})||_2\,||_{(1,0)}}|\textbf{Y}|_{H^2(K)}||\textbf{Y}||_{L^2(K)}
\end{equation*}

And consequently:
\begin{align*}
    \frac{d}{dt}|\textbf{Y}|_{H^2(K)}^2 \leq & (2||\textbf{B}||_{W^{0,\infty}}+4||\textbf{B}||_{W^{1,\infty}})|\textbf{Y}|_{H^2(K)}^2\\
    &+2||\textbf{B}||_{W^{1,\infty}}e^{2||\,||D(\Tilde{\textbf{V}})||_2\,||_{(1,0)}}|\textbf{Y}|_{H^1(K)}^2\\
    &+4\sqrt{2} (e^{2\sqrt{2}||\,||D(\Tilde{\textbf{V}})||_2\,||_{(1,0)}}||\,||\nabla_xD(\Tilde{\textbf{V}})||_2\,||_{(1,0)}||\textbf{B}||_{W^{1,\infty}}\\
    &+e^{2||\,||D(\Tilde{\textbf{V}})||_2\,||_{(1,0)}}||\textbf{B}||_{W^{2,\infty}})||\textbf{Y}||_{L^2(K)}|\textbf{Y}|_{H^2(K)}\\
    \leq & (2||\textbf{B}||_{W^{0,\infty}}+(4+2\sqrt{2}||\,||\nabla_xD(\Tilde{\textbf{V}})||_2\,||_{(1,0)})||\textbf{B}||_{W^{1,\infty}}\\
    &+2\sqrt{2}||\textbf{B}||_{W^{2,\infty}})|\textbf{Y}|_{H^2(K)}^2+2||\textbf{B}||_{W^{1,\infty}}e^{2||\,||D(\Tilde{\textbf{V}})||_2\,||_{(1,0)}}|\textbf{Y}|_{H^1(K)}^2\\
    &+2\sqrt{2} e^{4\sqrt{2}||\,||D(\Tilde{\textbf{V}})||_2\,||_{(1,0)}}(||\,||\nabla_xD(\Tilde{\textbf{V}})||_2\,||_{(1,0)}||\textbf{B}||_{W^{1,\infty}}\\
    &+||\textbf{B}||_{W^{2,\infty}})||\textbf{Y}||_{L^2(K)}^2
\end{align*}

And by using \eqref{apdx:Y1}, \eqref{apdx:Y2} and Gronwall's inequality:
\begin{multline*}
    |\textbf{Y}|_{H^2(K)}^2  \leq|K| \exp(6||\textbf{B}||_{(1,0)}+[5 +2\sqrt{2}\,||\,||\nabla_x D(\Tilde{\textbf{V}})||_2\,||_{(1,0)}]||\textbf{B}||_{(1,1)}\\
    +2\sqrt{2} ||\textbf{B}||_{(1,2)}+ 4\sqrt{2}\, ||\, ||D(\Tilde{\textbf{V}})||_2 \, ||_{(1,0)})
\end{multline*}

\end{proof}

\end{appendices}


\bibliography{sn-bibliography}

\end{document}